\documentclass{article}
\usepackage{amssymb}
\usepackage{eurosym}
\usepackage{ulem}
\usepackage{amsthm}
\usepackage{graphicx}
\usepackage{float}
\usepackage{amsmath}
\usepackage{hyperref}
\usepackage[T1]{fontenc}
\usepackage[utf8]{inputenc}
\usepackage{authblk}
\usepackage{xcolor}
\usepackage{multicol}

\newcommand{\mg}{\mathfrak g }

\newcommand{\mn}{\mathfrak n }
\newcommand{\mm}{\mathfrak m }
\newcommand{\mmu}{\mathfrak u }

\newcommand{\mk}{\mathfrak k }

\newcommand{\ma}{\mathfrak a }
\newcommand{\mgg}{\mathfrak g }
\newcommand{\mpp}{\mathfrak p }
\newcommand{\mt}{\mathfrak t }

\newcommand{\so}{\mathfrak{so} }

\newcommand{\gl}{\mathfrak{gl} }

\newcommand{\lela}{\left \langle}
\newcommand{\rira}{\right \rangle}
\newcommand{\lra}{\longrightarrow}

\newcommand{\bs}{\backslash}

\renewcommand{\l}{\lambda}

\newcommand{\GL}{\rm GL}

 \newcommand{\R}{\mathbb R}

\newcommand{\K}{\mathbb K}
\newcommand{\F}{\mathbb F}

\newtheorem{theorem}{Theorem}[section]
\newtheorem{lemma}[theorem]{Lemma}

\newtheorem{proposition}[theorem]{Proposition}

\newtheorem{example}[theorem]{Example}
\newtheorem{observation}[theorem]{Remark}
\newenvironment{prova}{\noindent {\bf Proof:}}{\hfill $\qed $ \newline}

\DeclareMathOperator{\ad}{ad}

\DeclareMathOperator{\diag}{diag}

\numberwithin{equation}{section}
\begin{document}

\title{Invariant almost complex structures on real flag manifolds}
\author[1]{Ana P. C. Freitas \thanks{Supported by CNPq grant 476024/2012-9. E-mail: anapcdfreitas@gmail.com}}
\author[1,2]{Viviana del Barco\thanks{Supported by Fapesp grant 2015/23896-5. E-mail: delbarc@ime.unicamp.br}}
\author[1]{Luiz A. B. San Martin \thanks{Supported by CNPq grant no. 476024/2012-9 and Fapesp grant no.
2012/18780-0. E-mail: smartin@ime.unicamp.br}}
\affil[1]{IMECC-UNICAMP}\affil[2]{UNR-CONICET}

\renewcommand\Authands{ and }

\maketitle

\begin{abstract}
In this work we study the existence of invariant almost complex structures
on real flag manifolds associated to split real forms of complex simple Lie
algebras. We show that, contrary to the complex case where the invariant
almost complex structures are well known, some real flag manifolds do not
admit such structures. We check which invariant almost complex structures
are integrable and prove that only some flag manifolds of the Lie algebra $%
C_{l}$ admit complex structures.
\end{abstract}


\section{Introduction}

A flag manifold of a non compact semisimple Lie algebra $\mathfrak{g}$, is a quotient space $\mathbb{%
F}_{\Theta }=G/P_{\Theta }$, where $G$ is a connected group with Lie algebra 
$\mathfrak{g}$ and $P_{\Theta }$ is a parabolic subgroup. If $K\subset G$ is
a maximal compact subgroup and $K_{\Theta }=K\cap P_{\Theta }$, then the
flag $\mathbb{F}_{\Theta }$ can be written in the form $\mathbb{F}_{\Theta
}=K/K_{\Theta }$.

In this work, we study the existence and integrability of invariant almost
complex structures on real flag manifolds $\mathbb{F}_{\Theta }$ in the case that $\mathfrak{g}$ is a split real form of a complex simple Lie algebra.  Our goal is to make an exhaustive investigation of the real flag manifolds $%
\mathbb{F}_{\Theta }$ that admit $K$-invariant almost complex structures and
to verify their integrability, that is, when they are indeed complex
structures.

The invariant geometry of complex flag manifolds has been extensively studied. Regarding invariant geometry of complex flag manifolds, the literature is exhaustive and
 goes back to Borel \cite{bor} and Wolf-Gray \cite{gw}, 
\cite{gw1}. Recent works are \cite{grego}, \cite{Monk}, \cite%
{burstall1990twistor}, \cite{burstall1987tournaments}, \cite%
{negreiros1988some}, \cite{san2003invariant}, \cite{san2006invariant}, \cite%
{grama2015invariant}, \cite{parede}, \cite{cohen20021} and \cite%
{arvanitoyeorgos1993new}.

For real flag manifolds the literature is much more sparse. There is no
systematic treatment of the invariant geometric structures on these flag
manifolds. An attempt to fill this gap was made recently by Patr\~{a}o and San
Martin \cite{patrao2015isotropy} who provide a detailed analysis of the
isotropy representations for the flag manifolds of the split real forms of
the complex simple Lie algebras.

In this paper we rely on the results of \cite{patrao2015isotropy} to build
(or to prove the non-existence of) $K$-invariant almost complex structures on
the real flag manifolds. The conclusion is that only a few flag manifolds
(associated to split real forms) admit $K$-invariant almost complex
structures. In this sense we obtain the following result.

\begin{theorem}\label{teo1}
A real flag manifold $\F_\Theta=K/K_\Theta$ admits a $K$-invariant almost complex structure structure if and only if it is a maximal flag of type $A_3$, $B_2$, $G_2$, $C_l$ for $l$ even or $D_l$ for $l\geq 4$, or if it is one of the following intermediate flags:
\begin{itemize}
\item of type $B_3$ and $\Theta=\{\lambda_1-\lambda_2, \lambda_2-\lambda_3\}$;
\item of type $C_l$ with  $\Theta=\{\lambda_{d}-\lambda_{d+1},\ldots,\lambda_{l-1}-\lambda_l,2\lambda_l\}$ for $d>1$, $d$ odd.
\item of type $D_l$ with $l=4$ and $\Theta$ being one of: $\{ \lambda_{1}-\lambda_{2},\lambda_{3}-\lambda_{4}\}$, $\{ \lambda_{1}-\lambda_{2},\lambda_{3}+\lambda_{4}\}$, $\{ \lambda_{3}-\lambda_{4},\lambda_{3}+\lambda_{4}\}$.
\end{itemize}\end{theorem}

The next step is to check which of the existing almost complex structures
are integrable. By making computations with the Nijenhuis tensor we arrive
at the following result.

\begin{theorem}\label{teo2}
A real flag manifold $\F_\Theta=K/K_\Theta$ admits $K$-invariant complex structures if and only if it is of type $C_l$ and $\Theta=\{\l_d-\l_{d+1},\ldots,\l_{l-1}-\l_l,2\l_l\}$ with $d>1$, $d$ odd.
\end{theorem}

These complex flag manifolds are realized as manifolds of flags $\left( V_{1}\subset
\cdots \subset V_{k}\right) $ of subspaces of $\mathbb{R}^{2l}$ that are
isotropic with respect to the standard symplectic form of $\mathbb{R}^{2l}$. Moreover $\F_\Theta$ is finitely covered by $U(l)/U(l-d)$ and the complex structures on $\F_\Theta$ can be lifted to this covering space.
 
To prove the results above we mainly use the isotropy decomposition of $T_{b_\Theta}\F_\Theta$, the tangent space of the flag a the
origin $b_\Theta$. In \cite{patrao2015isotropy} there are described the $K_\Theta$-invariant and irreducible components of this representation obtaining a decomposition
\begin{equation*}
T_{b_{\Theta }}\mathbb{F}_{\Theta }=V_{1}\oplus \ldots \oplus V_{k}.
\end{equation*}%
This decomposition is essential to find $K$-invariant geometries on $\mathbb{F}_{\Theta }$. It is well known that the compact isotropy group is a product $K_\Theta=M(K_\Theta)_0$ where $M$ is the isotropy of the maximal flag and $(K_\Theta)_0$ the connected component of the identity. An almost complex structure commutes with the isotropy representation of $K_\Theta$ if and only if it commutes with the $M$ and $(K_\Theta)_0$ representations on the  tangent space. This allows us to split the proofs  in two stages: study $M$-invariance on the one hand, and the condition of commutativity with $\ad_X$ for all $X\in \mk_\Theta = \mathrm{Lie}(K_\Theta)$, on the other hand.

A necessary and sufficient condition for a real flag to admit $M$-invariant almost
complex structures is that every $M$-equivalence class on $\Pi ^{+}\setminus
\langle \Theta \rangle^+ $ has an even number of elements. Two roots $\alpha $
e $\beta $ lie in the same $M$-equivalence class if the representations of $%
M $ on $\mathfrak{g}_{\alpha }$ and $\mathfrak{g}_{\beta }$ are equivalent.
This condition is necessary for $\F_\Theta$ to admit $K_\Theta$ invariant almost complex structures, so by inspection of these equivalence classes we discard many flags manifolds. For the remaining cases
we focus on the $\mk_\Theta$ representation on $T_{b_\Theta}\F_\Theta$. We should remark that in all cases we give the almost complex structures explicitly, in a constructive way. Integrability is proved by computing the Nijenhuis tensor.

It is worth stressing a main difference in the isotropy representation of $%
K_{\Theta}$ between the real case and the complex case. In the real flag,
occur cases where two $K_{\Theta}$-invariant and irreducible components are
equivalent. In the complex case this fact does not occur. Consequently, on
the complex case, the $K_{\Theta}$-invariant and irreducible components, in
the isotropy representation of $\mathbb{F}_{\Theta}$, are invariant by
almost complex structures. On the real flag, occur cases where $JV_i=V_j$,
for $V_i$ and $V_j$ equivalent $K_{\Theta}$-invariant and irreducible
components.

This work is organized of following manner. In Section 2 we fix  notations and present the first results on existence of $M$-invariant complex structures. We give necessary and sufficient conditions for a flag manifold to admit such structure. In the case of a maximal flag, that is $\Theta=\emptyset$, this is all we need to pursue our study since $K_\Theta=M$. Section 3 focuses in this case. Section 4 deals with intermediate flags, that is $\Theta\neq \emptyset$. We only consider those intermediate flags verifying the necessary condition of Section 2. The full  comprehension of the isotropy representation of $K_\Theta$ is needed, so we fully describe it for the cases under study. The propositions in Sections 3 and 4 account to Theorems \ref{teo1} and \ref{teo2} above.

\section{Notation and preliminary results}

We refer to \cite{sangrupos,knapp2013lie} for further developments of the concepts in this section. We assume throughout the paper that $\mathfrak{g}$ is the split real form of a complex simple Lie algebra $\mathfrak{g}_{\mathbb{C}}$.
If $\mathfrak{g}=\mathfrak{k}\oplus \mathfrak{a}%
\oplus \mathfrak{n}$ is an Iwasawa decomposition then $\mathfrak{a}$ is a
Cartan subalgebra. Denote $\Pi$ the set of roots of $\mgg$ associated to $\ma$. If $\alpha \in \mathfrak{a}^{\ast }$ is a root then we write 
\begin{equation*}
\mathfrak{g}_{\alpha }=\{X\in \mathfrak{g}:\mathrm{ad}\left( H\right)
X=\alpha \left( H\right) X,~H\in \mathfrak{a}\}
\end{equation*}%
for the corresponding root space, which is one-dimensional since $\mg$ is split. Let $\Pi^+$ be a set of positive roots and $\Sigma$ the corresponding positive simple roots.

The set of parabolic Lie subalgebras of $\mg$ is parametrized by the subsets of simple roots $\Sigma$. Given $\Theta\subset \Sigma$, the corresponding parabolic subalgebra is given by 
$$\mpp_\Theta=\ma\oplus \sum_{\alpha \in \Pi^+} \mg_\alpha \oplus \sum_{\alpha\in \lela\Theta\rira^-}\mg_\alpha=\ma \oplus \sum_{\alpha \in \lela\Theta\rira^+\cup \lela\Theta\rira^-} \mg_\alpha \oplus \sum_{\alpha\in \Pi^+\bs\lela\Theta\rira^+}\mg_\alpha$$ where $\lela\Theta\rira^\pm$ is the set of positive/negative roots generated by $\Theta$. 

Denote by $G$ the group of inner automorphisms of $\mg$, which is connected and generated by $\exp \ad(\mgg)$ inside $\GL(\mg)$. Let $K$ be the maximal compact subgroup of $G$, then $K$ is generated by $\ad(\mk)$. The standard parabolic subgroup $P_\Theta$ of $G$ is the normalizer of $\mpp_\Theta$ in $G$. The associated flag manifold is defined by $\F_\Theta=G/P_\Theta$. The compact subgroup $K$ acts transitively on $\F_\Theta$ so we obtain $\F_\Theta=K/K_\Theta$ where $K_\Theta=K\cap P_\Theta$. Fixing an origin $b_\Theta$ in $\F_\Theta$ we identify the tangent space $T_{b_\Theta}\F_\Theta$ with the nilpotent Lie algebra 
$$\mn_\Theta^-= \sum_{\alpha\in \Pi^-\bs\lela\Theta\rira^-}\mg_\alpha.$$

In $\mn^-$, the isotropy representation of $K_\Theta$ on  $T_{b_\Theta}\F_\Theta$ is just the adjoint representation, since $\mn^-_\Theta$ is normalized by $K_\Theta$. The Lie algebra $\mk_\Theta$ of $K_\Theta$ is $$\mk_\Theta=\sum_{\alpha \in \lela\Theta\rira^+\cup \lela\Theta\rira^-} (\mg_\alpha\oplus\mgg_{-\alpha})\cap \mk.$$ Compactness of $K$ implies that $\mk_\Theta$ admits a reductive complement $\mk_\Theta$ so that $\mk=\mk_\Theta\oplus \mm_\Theta$ and $T_{b_\Theta}\F_\Theta$ is identified also with $\mm_\Theta$. The map $X_\alpha\lra X_{\alpha}-X_{-\alpha}$ for $\alpha \in \Pi^-\bs \lela\Theta\rira^-$ is a $K_\Theta$ invariant map from $\mn_\Theta^-$ to $\mm_\Theta$. Along the paper we will call isotropy representation either the representation of $K_\Theta$ on $\mn_\Theta^-$ or on $\mm_\Theta$, without making any difference or special mention. In some cases we will even use $\mn_\Theta^+$ instead of $\mn_\Theta^-$.

Let $M$ be the centralizer of $\ma$ in $K$. Then $K_\Theta=M\cdot (K_\Theta)_0$ where $(K_\Theta)_0$ is the connected component of the identity of $K_\Theta$. Thus $M$ acts on $T_{b_{\Theta }}\mathbb{F}_{\Theta }$ by restricting the isotropy representation of $K_\Theta$. The group $M$ is finite and acts on $\mn_\Theta^-$ leaving each root space $\mg_\alpha$ invariant. Moreover if $m\in M$ and $X\in \mgg_\alpha$ then $\mathrm{Ad}(m)X=\pm X$. Two roots $\alpha$ and $\beta$ are called $M$-equivalent, which we will
denote by $\alpha\sim_M \beta$, if the representations of $M$ on the root
spaces $\mg_\alpha$ and $\mg_\beta$ are equivalent. The $M$-equivalence classes were described in \cite{patrao2015isotropy}.

When $\Theta=\emptyset$, we drop all the sub indexes $\Theta$. The associated flag manifold is the maximal flag $\F=K/M$ and the tangent space at the origin $b$ will be identified with $\mn^-$.\smallskip


Let $U$ be a group of linear maps of the vector space $V$. A subspace $W\subset U$ is $U$-invariant if $ux\in W$ for all $x\in W$. A complex structure on $V$ is endomorphism $J:V\lra V$ such that $J^2=-1$ and it is said to be $U$-invariant if $uJ=Ju$ for all $u\in U$. We shall prove two technical results.

\begin{lemma}
\label{lem1} Let $W\subset V$ be a $U$-invariant space. Then the following statements are true:

\begin{enumerate}
\item $JW$ is $U$-invariant as well.
\item $W$ is irreducible if and only if $JW$ is irreducible.
\item The representations of $U$ on $W$ and $JW$ are equivalent.
\item\label{Wirred} If $W$ is irreducible then either $W\cap JW=\{0\}$ or $JW=W$.
\item If $\dim W=1$ then $W\cap JW=\{0\}$.
\end{enumerate}
\end{lemma}

\begin{proof}
Take $u\in U$ and $x\in W$. Then, $uJx=Jux\in JW$ showing that $JW$ is $U$%
-invariant.

Suppose that $W$ is irreducible and let $A\subset JW$ be a $U$-invariant
subspace. Then $J^{-1}A=JA\subset W$ is also $U$-invariant. Hence, $JA=W$ or 
$JA=\{0\}$, which implies that $A=JW$ or $A=\{0\}$. Thus $JW$ is irreducible.

As $J$ commutes with the elements of $U$, the map $J:W\rightarrow JW$
intertwines the representations on $W$ and $JW$ so that they are equivalent.
Since $W\cap JW\subset W$ is $U$-invariant and $W$ is irreducible we get
item 4. Finally $W\cap JW=\{0\}$ if $\dim W=1$  because the eigenvalues of $%
J$ are $\pm i$ hence $W$ is not invariant by $J$.
\end{proof}

\begin{lemma}\label{lem2} Let $W_i$, $i=1,2$ be $U$-invariant and irreducible subspaces of $V$ such that $W_1\cap W_2=0$ and the representations of $U$ on $W_1$ is not equivalent to that on $W_2$. If $V=W_1\oplus W_2\oplus W$ for some complementary subspace $W$ and $J$ is a $U$-invariant complex structure, then $Jw_1\in W_1\oplus W$ for all $w_1\in W_1$.
\end{lemma}

\begin{proof}
Consider $P:V\lra W_2$ the projection map with respect to the decomposition above. The map $P\circ J:W_1\lra W_2$ is $U$-invariant and bijective if non-zero, since its domain and target spaces are irreducible. Thus it is an equivalence between the representations of $U$, if non-zero. Therefore, $P\circ J\equiv 0$ and the result follows.
\end{proof}
Under the of the lemma above, in the particular case of $V=W_1\oplus W_2$ we have $JW_i=W_i$, $i=1,2$.

\smallskip
From the general theory of invariant tensors on homogeneous manifolds we know that $K$-invariant almost complex structure on the flag manifold $\F_\Theta=K/K_\Theta$ are in one to one correspondence with $K_\Theta$-invariant complex structures $J:T_{b_{\Theta }}\mathbb{F}_{\Theta
}\rightarrow T_{b_{\Theta }}\mathbb{F}_{\Theta }$. Recall that $ T_{b_{\Theta }}\mathbb{F}_{\Theta }$ identifies with $\mn_\Theta^-$ (or $\mm_\Theta$) and this identification preserves the $K_\Theta$ representation. So $K$-invariant almost complex structures on $\F_\Theta$ also correspond to $K_\Theta$-invariant complex structures on $\mn_\Theta^-$. 

Let $J:\mn_\Theta^-\lra \mn_\Theta^-$ be a complex structure and assume it is only $M$-invariant. Since $K_\Theta=M(K_\Theta)_0$ we have that $J$ is also  $K_\Theta$-invariant if and only if $J$ commutes with the elements in $(K_\Theta)_0$, or equivalently, $\ad_XJ=J\ad_X$ for all $X\in \mk_\Theta$ (because of connectedness).

\begin{proposition}
\label{propminvar} Let $\mathbb{F}_{\Theta }$ be a real flag manifold
associated to a split real form. Then a necessary and sufficient condition for the existence
of a $M$-invariant complex structure $J:T_{b_{\Theta }}\mathbb{F}_{\Theta
}\rightarrow T_{b_{\Theta }}\mathbb{F}_{\Theta }$ is that the number of
elements in each $M$-equivalence class $\left[ \alpha \right] $ in $\Pi
^{-}\setminus \langle \Theta \rangle ^-$ is even.

In this case the $M$-invariant complex structures are given by direct sums
of invariant structures on the subspaces $V_{\left[ \alpha \right]
}=\sum_{\beta \sim _{M}\alpha }\mathfrak{g}_{\beta }\subset \mathfrak{n}%
_{\Theta }^{-}$. In a subspace $V_{\left[ \alpha \right] }$ the set of $M$%
-invariant structures is parametrized by $\mathrm{Gl}(d,\mathbb{R})/\mathrm{%
Gl}(d/2,\mathbb{C})$, where $d=\dim V_{\left[ \alpha \right] }$.
\end{proposition}

\begin{prova}
If $\alpha \in \Pi ^{-}\backslash \langle \Theta \rangle ^-$ then $\mathfrak{g}%
_{\alpha }\in \mathfrak{n}_{\Theta }^{-}$ and $\dim \mathfrak{g}_{\alpha }=1$
(because $\mathfrak{g}$ is a split real form). The subspace $J\mathfrak{g}%
_{\alpha }\subset \mathfrak{n}_{\Theta }^{-}$ is different of $\mathfrak{g}%
_{\alpha }$ by 5. in Lemma \ref{lem1} and the representation of $M$ in $J%
\mathfrak{g}_{\alpha }$ is equivalent to representation in $\mathfrak{g}%
_{\alpha }$. Lemma \ref{lem2} implies that $J\mathfrak{g}_{\alpha }$ is contained in the
subspace $V_{\left[ \alpha \right] }=\sum_{\beta \sim _{M}\alpha }\mathfrak{g%
}_{\beta }$. Applying the same argument to the roots $\beta $ that are $M$%
-equivalent to $\alpha $, we obtain $JV_{\alpha }=V_{\alpha }$. As $J^{2}=-1$, it follows that $\dim V_{\alpha }$ is even and, hence, the
number of roots $M$-equivalent to $\alpha $ is even. This proves that the condition is necessary.

To see the sufficiency take a $M$-equivalent class $[\alpha ]$ so that by
assumption the subspace $V_{[\alpha ]}=\sum_{\beta \sim _{M}\alpha }%
\mathfrak{g}_{\beta }$ is even dimensional. Given $m\in M$ we have $\mathrm{Ad}\left( m\right) X=\pm X$
if $X$ belongs to a root space $X\in \mathfrak{g}_{\beta }$. In this
equality the sign does not change when $\beta $ runs through a $M$%
-equivalence class. It follows that $\mathrm{Ad}\left( m\right) =\pm 1$ on $V_{[\alpha ]}$. Hence any complex structure on $V_{[\alpha ]}$ is $%
M $-invariant. Taking direct sums of complex structures on the several $V_{%
\left[ \alpha \right] }$ we get $M$-invariant complex structures on $%
T_{b_{\Theta }}\mathbb{F}_{\Theta }\simeq \mathfrak{n}_{\Theta }^{-}$.

Finally the set of complex structures in a $d$-dimensional real space ($d$
even) is parametrized by $\mathrm{Gl}(d,\mathbb{R})/\mathrm{Gl}(d/2,\mathbb{C%
})$.
\end{prova}

We use the results in \cite{patrao2015isotropy} to present in Table \ref{table.Minv} all possible subsets $\Theta\subset \Sigma$ for which the $M$-equivalence classes in $\Pi ^{-}\setminus \langle
\Theta \rangle^-$ have an even amount of elements. Even though we do not give the explicit computations to construct this table, we present the $M$-equivalence classes for some cases in the followings sections.

\begin{table}[h]
\centering
\begin{tabular}{|c|c|}
\hline
Type & $\Theta$\\
\hline
$A_3$& $\emptyset$\\
\hline
$B_2$ & $\emptyset$\\
 \hline
$B_3$ & $\{\lambda_1-\lambda_2,\lambda_2-\lambda_3\}$\\
\hline
$C_4$ & $\emptyset$,$\{\lambda_1-\lambda_2,\lambda_3-\lambda_4\}$,$\{\lambda_3-\lambda_4,2\lambda_4\}$\\
 \hline
$C_l$, $l\neq 4$ & $\emptyset$ only for $l$ even,  \\
&$\{\lambda_d-\lambda_{d+1},\cdots, \lambda_{l-1}-\lambda_{l},2\lambda_l\}$,  $1< d\leq l-1$, $d$ odd, for all $l$ \\
\hline
$D_4$ & $\emptyset$, $\{ \lambda_{1}-\lambda_{2},\lambda_{3}-\lambda_{4}\}$, $\{ \lambda_{1}-\lambda_{2},\lambda_{3}+\lambda_{4}\}$,\\
      &$\{ \lambda_{3}-\lambda_{4},\lambda_{3}+\lambda_{4}\}$, $\{\lambda_1-\lambda_{2}, \lambda_{2}-\lambda_{3},\lambda_{3}-\lambda_{4}\}$\\
      &  $\{\lambda_1-\lambda_{2}, \lambda_{2}-\lambda_{3},\lambda_{3}+\lambda_{4}\}$, $\{\lambda_2-\lambda_{3}, \lambda_{3}-\lambda_{4},\lambda_{3}+\lambda_{4}\}$\\
      \hline
$D_l$, $l\geq 5$ & $\emptyset$, $\{\lambda_d-\lambda_{d+1},\cdots, \lambda_{l-1}-\lambda_{l},\lambda_{l-1}+\lambda_{l}\}$, $1<d\leq l-1$.\\
\hline
$G_2$& $\emptyset$\\
\hline
\end{tabular}
\caption{$M$-equivalence classes in $\Pi ^{-}\setminus \langle
\Theta \rangle^- $ with even elements}
\label{table.Minv}
\end{table}

Complex structures on $\F_\Theta$ which are invariant under $K$ are induced by $K_\Theta$-invariant complex structures on the tangent space and, in particular, are $M$-invariant. Hence Proposition \ref{propminvar} and a simple inspection of Table \ref{table.Minv} give the following result.
\begin{proposition}
\label{corminvar} Let $\mathbb{F}_{\Theta }$ be a real flag manifold
associated to a split real form. If $\mathbb{F}_{\Theta }$ admits a $K$%
-invariant almost complex structure, then $\Theta$ is in Table \ref{table.Minv}.
\end{proposition}

An invariant complex structure $J:\mn_\Theta^-\lra \mn_\Theta^-$ induced is integrable if the Nijenhuis tensor vanishes, that is if
$$N_J(X,Y):=[JX,JY]-[X,Y]-J[JX,Y]-J[X,JY]=0,\; \mbox{ for all }X,Y \in \mn_\Theta^-. $$

\section{$K$-invariant complex structures on maximal flags}
For a maximal flag manifold the isotropy subgroup $\K_\Theta$ is the centralizer of $\ma$ inside $K$, that is, $\K_\Theta=M$. Hence Proposition \ref{propminvar} solves the question of existence of almost complex structures, remaining only integrability to be solved.  The main result of this section is the following.

\begin{proposition}\label{cormaxflag}
The maximal real flag $\F$ associated to a split real form admits a $K$-invariant almost complex structure if and only if $\F$ is of type $A_3$, $B_2$,  $G_2$, $C_l$ for  even $l$ and $D_l$ for $l\geq 4$. None of these structures is integrable.
\end{proposition}
\begin{proof} 
By Proposition \ref{propminvar}, a maximal flag $\F $ admits an $M$-invariant almost complex structure if and only if it appears in Table \ref{table.Minv}. 

Recall that an almost complex structure $J:\mn^-\lra \mn^-$ is a sum of almost complex structures $J_{[\alpha]}:V_{[\alpha]}\lra V_{[\alpha]}$ for $\alpha\in \Pi^-$ gives an $M$-invariant almost complex structure in $\F$. We address integrability of these structures by fixing one of these $J:\mn^-\lra \mn^-$ and we study case by case.  

Notice that if $V_{[\alpha]}$ is two dimensional with basis $\mathcal B$, then the matrix of $J_{[\alpha]}$ in $\mathcal B$ is
\begin{equation}\label{eq.jota}
\left(\begin{matrix}
a&\frac{-(1+a^2)}{c}\\
c&-a
\end{matrix}\right),\mbox{ with} \, a,c\in \R, \,c\neq 0.
\end{equation}

\begin{itemize}
\item Case $A_3$.
The $M$-equivalence classes of negative roots are: 
\begin{equation*}
\{\lambda_2-\lambda_1,\lambda_4-\lambda_3\},\
\{\lambda_3-\lambda_1,\lambda_4-\lambda_2\}\ \mbox e \mbox \ %
\{\lambda_4-\lambda_1,\lambda_3-\lambda_2\}.
\end{equation*}

Thus for $i=2,3,4$, $\dim V_{[\l_{i}-\l_1]}=2$  and it is spanned by $\{E_{i1},E_{st}\}$ with $s>t$, $\{s,t\}\cap \{i,1\}=\emptyset$ and $\{s,t\}\cup \{i,1\}=\{1,\ldots,4\}$; here $E_{jk}$ is the $4\times 4$ matrix with 1 in the $jk$ entry and zero elsewhere. For $i=2,3,4$, let $a_i,c_i\in \R$ such that $J|_{V_{[\l_{i}-\l_1]}}$ in this basis has the following form
$$\left(\begin{matrix}
a_i&\frac{-(1+a_i^2)}{c_i}\\
c_i&-a_i
\end{matrix}\right),\quad c_i\neq 0.$$

Explicit computations give
\begin{eqnarray*}
N_J(E_{21},E_{31}) &=&(c_3-c_2)c_4E_{32}+(c_2a_3-a_2c_3+a_4(c_3-c_2))E_{41},\\
N_J(E_{21},E_{41}) &=&c_4(a_3-a_2)E_{31}+c_4(c_2+c_3) E_{42}.
\end{eqnarray*} These two equations cannot be zero simultaneously since $c_i\neq 0$. Thus the Nijenhuis does not vanish and $J$ is not integrable.

\item Case $B_2$. The $M$-equivalence classes of negative roots are 
\begin{equation*}
\{\lambda_2-\lambda_1,-\lambda_2-\lambda_1 \}\ \mbox e \mbox \ %
\{-\lambda_1,-\lambda_2\}.
\end{equation*}

Let $X_{21}$, $Y_{21}$, $X_{1}$ and $X_{2}$ be elements of a Weyl basis generating $\mathfrak{g}_{\lambda_2-\lambda_1}$, $\mathfrak{g}%
_{-\lambda_2-\lambda_1}$, $\mathfrak{g}_{-\lambda_1}$ and $\mathfrak{g}%
_{-\lambda_2}$, respectively.
Thus $J$ verifies
\begin{equation*}
\begin{array}{ll}
JX_{21} =a_{21}X_{21}+c_{21}Y_{21},  & JX_{1} =a_1X_1+c_{1}X_{2}, \\ 
JY_{21}=-(1+a_{21}^2)X_{21}/c_{21}-a_{21}Y_{21}, &JX_{2}= -(1+a_1^2)X_{1}/c_{1}-a_1X_2,
\end{array} 
\end{equation*}
 with $c_1,c_{21}\neq 0$. 
 
 Let $m=m_{\lambda_2-\lambda_1,-\lambda_2}\neq 0$  be the corresponding coefficient in the Weyl basis, that is, $[X_{21},X_2]=m X_{1}$. Then
\begin{equation*}
\begin{array}{ccl}
N_J(X_{21},X_{1}) & = & [JX_{21},JX_{1}]- [X_{21},X_{1}]-J[X_{21},JX_{1}]
-J[JX_{21},X_{1}] \\ 
& = & -m c_{1}^2X_{2} +mc_1(a_{21}-a_1)X_1 
\end{array}%
\end{equation*}
which is never zero since $mc_1^2\neq 0$. Therefore $J$ is not integrable.

\item Case $C_4$.  The $M$-equivalence classes are: 
\begin{eqnarray*}&
\{\pm\lambda_2-\lambda_1,\pm\lambda_4-\lambda_3\},
\{\pm\lambda_3-\lambda_1,\pm\lambda_4-\lambda_2\},
\{\pm\lambda_4-\lambda_1,\pm\lambda_3-\lambda_2\},&\\
&\{-2\lambda_i:\ i=1,\ldots, 4\}.&
\end{eqnarray*}
Notice that $\dim V_{[2\l_1]}=\dim V_{[\l_{i}-\l_1]}=4$ for $i=2,3,4$.
Let $(a_{ij})_{ij}$,  $(b_{ij})_{ij}$, $(c_{ij})_{ij}$ be the matrices corresponding to $J|_{V_{[\l_2-\l_1]}}$, $J|_{V_{[\l_3-\l_1]}}$, $J|_{V_{[\l_4-\l_1]}}$ respectively in a Weyl basis of $\mn^-$.

Then
$N_J(X_{-\l_2-\l_1},X_{-2\l_2})=0$ and $N_J(X_{-\l_4-\l_3},X_{-2\l_4})=0$ imply $a_{12}=a_{34}=0$ and moreover $a_{14}^2+a_{24}^2\neq 0$ because otherwise $X_{-\l_4-\l_3}$ would be an eigenvector of $J$.  Analogously we obtain $b_{12}=b_{34}=c_{12}=c_{34}=0$ and $b_{14}^2+b_{24}^2\neq 0$, $c_{14}^2+c_{24}^2\neq 0$.

With these conditions, $N_J(X_{-\l_2-\l_1},X_{-2\l_4})=0$ imply $a_{32}=0$ and $a_{42}\neq0$. Similar computations give $b_{32}=c_{32}=0$ and $b_{42}\neq 0$, $c_{42}\neq 0$. Now $J^2=-1$ imply $a_{14}=b_{14}=c_{14}=0$.

All this account to $N_J(X_{\l_2-\l_1},X_{-\l_3-\l_1})=0$ and $N_J(X_{\l_2-\l_1},X_{-\l_4-\l_1})=0$ only if, respectively,  $a_{31}=c_{42}$ and $a_{31}=-c_{42}$. This  clearly cannot hold since $c_{42}\neq 0$.

\item Case $C_l$, $l$ even and $l\geq 6$. The $M$-equivalence classes are 
\begin{equation*}
\{\pm\lambda_s-
\lambda_i\}, \;1\leq i<s\leq l,\ \mbox{ and }\ \{2\lambda_1,
\ldots,2\lambda_l\}.
\end{equation*}

Let $X_{si}$, $Y_{si}$ and $X_{j}$ be the
generators of the roots spaces $\mathfrak{g}_{\lambda_s-\lambda_i}$, $\mathfrak{g%
}_{-\lambda_s-\lambda_i}$ and $\mathfrak{g}_{-2\lambda_j}$,  respectively, corresponding to a Weyl basis. In this case we have $\dim V_{[\l_{s}-\l_i]}=2$ while $\dim V_{[2\l_1]}=l$, even. Thus $JX_{1} =\sum_{j=1}^l b_j X_{j}$ and for $s=1, \ldots,l$ we have
\begin{equation*}
JX_{s1} =a_{s1}X_{s1}+c_{s1}Y_{s1},  \quad
JY_{s1}=-\frac{(1+a_{s1}^2)}{c_{s1}}X_{s1}-a_{s1}Y_{s1}, \quad c_{s1}\neq 0 .
\end{equation*}

We compute the Nijenhuis tensor on the vectors $X_1$ and $X_{s1}$, for $s=2,\ldots,l$. Denote $m=m_{\l_s-\l_1,-2\l_s}\neq 0$, then we get 
\begin{eqnarray*}
N_J(X_{s1},X_{1}) & = & [JX_{s1},JX_{1}]- [X_{s1},X_{1}]-J[X_{s1},JX_{1}]
-J[JX_{s1},X_{1}] \\ 
&=& [a_{s1}X_{s1}+c_{s1}Y_{s1},\sum_{j=1}^l b_j X_{j}]-b_s m JY_{s1}\\
&=& a_{s1}b_s m Y_{s1}-b_s m (-\frac{(1+a_{s1}^2)}{c_{s1}}X_{s1}-a_{s1}Y_{s1})\\
&=&b_s m \frac{(1+a_{s1}^2)}{c_{s1}} X_{s1}+a_{s1}(b_sm+1) Y_{s1}.
\end{eqnarray*}
Hence $N_J(X_{s1},X_{1})=0$ if and only if $b_{s}m=0$. Thus $J$ integrable implies $b_s=0$ for $s=2,\ldots,l$. and therefore $JX_1=b_1X_1$, which contradicts the fact that $J^2=-1$. Thus $J$ is not integrable.

\item Case $D_4$. The $M$-equivalence classes are
\begin{equation*}
\{\pm\lambda_2-\lambda_1,\pm\lambda_4-\lambda_3\},
\{\pm\lambda_3-\lambda_1,\pm\lambda_4-\lambda_1\},
\{\pm\lambda_4-\lambda_1,\pm\lambda_3-\lambda_2\}.
\end{equation*}
Clearly, $\dim V_{[\l_{i}-\l_1]}=4$ for $i=2,3,4$. We proceed as in the $C_4$ case. Let $(a_{ij})_{ij}$,  $(b_{ij})_{ij}$, $(c_{ij})_{ij}$ be the matrices corresponding to $J|_{V_{[\l_2-\l_1]}}$, $J|_{V_{[\l_3-\l_1]}}$, $J|_{V_{[\l_4-\l_1]}}$, respectively, in a Weyl basis of $\mn^-$.

By imposing $N_J(X_\gamma,X_\delta)=0$ for $\gamma\in [\l_3-\l_1]$ and $\delta \in [\l_4-\l_1]$ we obtain that 
the matrix of $J|_{V_{[\l_4-\l_1]}}$ in the Weyl basis is
$$\left(\begin{matrix}
-b_{44} & -b_{34} & b_{24} & b_{14}\\
-b_{43} & -b_{33} & b_{23} & b_{13}\\
 b_{42} & b_{32} &-b_{22} &-b_{12}\\
 b_{41} & b_{31} &-b_{21} &-b_{11}
\end{matrix}\right).$$
With this, $N_J(X_{\l_2-\l_1},X_{-\l_4-\l_1})=0$, $N_J(X_{-\l_4-\l_3},X_{-\l_3-\l_1})=0$  and $N_J(X_{\l_4-\l_3},X_{-\l_3-\l_1})=0$ imply $b_{12}b_{32}=0$, $b_{12}b_{42}= 0$ and $b_{32}b_{42}= 0$. But we know that $a_{12}^2+a_{32}^2+a_{42}^2\neq 0$ since $X_{-\l_2-\l_1}$ is not an eigenvector. So we conclude that only one of $b_{12},b_{32},b_{42}$ is not zero. In each of the three cases we obtain $a_{12}=a_{32}=a_{42}=0$ if $N_J$ vanishes, which cannot happen since $X_{-\l_2-\l_1}$ is not an eigenvector of $J$.

\item Case $D_l$, $l\geq 5$. The $M$-equivalence classes are:
\begin{equation*}
\{\pm\lambda_j-\lambda_i\}, \quad 1\leq i<j\leq l.
\end{equation*}
For $1\leq i<j\leq l$, we have $\dim V_{[\l_{j}-\l_i]}=2$; let $X_{ij}$ be a generator of $\mgg_{\l_i-\l_j}$ and let $Y_{ij}$ be a generator of $\mgg_{\l_i+\l_j}$. Thus $V_{[\l_j-\l_i]}$ is spanned by $\{X_{ij},Y_{ij}\}$ and $J$ in this basis has a matrix of the form
\begin{equation*}
\left(\begin{matrix}
a_{ij}&\frac{-(1+a_{ij}^2)}{c_{ij}}\\
c_{ij}&-a_{ij}
\end{matrix}\right),\mbox{ where } c_{ij}\neq 0.
\end{equation*}

Conditions $N_J(X_{13},X_{23})=0$  and  $N_J(X_{12},X_{23})=0$ imply
\begin{equation}\label{eq.igualcm}
\frac{m_{\l_1-\l_2,\l_2+\l_3}}{m_{\l_1-\l_2,\l_2-\l_3}}=\frac{c_{13}}{c_{23}}=-\frac{m_{\l_1-\l_3,\l_2+\l_3}}{m_{\l_1+\l_3,\l_2-\l_3}}.
\end{equation}
Now using Jacobi identity, we have 
\begin{equation*}
\begin{array}{lll}
0 & = & 
[Y_{23},[X_{12},X_{23}]]-[[Y_{23},X_{12}],X_{23}]-[X_{12},[Y_{23},X_{23}]]
\\ 
& = & m_{\lambda_1-\lambda_2,\lambda_2-\lambda_3}[Y_{23},X_{13}]+m_{%
\lambda_2 +\lambda_3,\lambda_1-\lambda_2}[X_{23},Y_{13}] \\ 
& = & \left(
m_{\lambda_1-\lambda_2,\lambda_2-\lambda_3}m_{\lambda_2+\lambda_3,\lambda_1-%
\lambda_3}+m_{\lambda_2+\lambda_3,\lambda_1-\lambda_2}m_{\lambda_2-%
\lambda_3,\lambda_1+\lambda_3}\right)Y_{12}.%
\end{array}%
\end{equation*}
Thus 
\begin{eqnarray*}
m_{\lambda_1-\lambda_2,\lambda_2-\lambda_3}m_{\lambda_2+\lambda_3,\lambda_1-%
\lambda_3}&=&-m_{\lambda_2+\lambda_3,\lambda_1-\lambda_2}m_{\lambda_2-%
\lambda_3,\lambda_1+\lambda_3}\\
&=&-m_{\lambda_1-\lambda_2,\lambda_2+\lambda_3}m_{\lambda_1+\lambda_3,%
\lambda_2-\lambda_3},
\end{eqnarray*} and therefore
\begin{equation}
 \frac{m_{\l_1-\l_2,\l_2+\l_3}}{m_{\l_1-\l_2,\l_2-\l_3}}=\frac{m_{\l_1-\l_3,\l_2+\l_3}}{m_{\l_1+\l_3,\l_2-\l_3}}.
\end{equation}
This equation clearly contradicts \eqref{eq.igualcm} and hence $J$ is not integrable.

\item Case $G_2$. The $M$-equivalence classes are
\begin{equation*}
\{-\l_1, -2\l_2-\l_1\},\ \{-
\l_2-\l_1,-3\l_2-\l_1\},\ \{ -\l_2,-3\l_2-2\l_1\}.
\end{equation*}

For $(i,j)\in\{(1,0),(0,1),(1,1)\}$, $\dim V_{[-i\l_1-j\l_2]}=2$.
In  a Weyl basis of $\mn^-$ we have that the matrix of $J|_{V_{[-i\l_1-j\l_2]}}$ has the form
\begin{equation*}
\left(\begin{matrix}
a_{ij}&\frac{-(1+a_{ij}^2)}{c_{ij}}\\
c_{ij}&-a_{ij}
\end{matrix}\right),\mbox{ where } c_{ij}\neq 0.
\end{equation*}
Denote $m=m_{-(\l_1+\l_2),-\l_2}$ then
\begin{eqnarray*}
N_J(X_{-\l_1-\l_2},X_{-\l_2}) & = &m(a_{11}a_{01}-1)X_{-\l_1-2\l_2}-m(a_{11}+a_{01})JX_{-\l_1-2\l_2}\\
&=&m\left((a_{11}a_{01}-1)+a_{10}(a_{11}+a_{01})\right)X_{-\l_1-2\l_2}\\
&&\qquad+m(a_{11}+a_{01})\frac{1+a_{10}^2}{c_{10}}X_{-\l_1}.
\end{eqnarray*}
Thus
$$N_J(X_{-\l_1-\l_2},X_{-\l_2})=0 \;\Leftrightarrow\;
a_{01}=-a_{11}\,\mbox{ and } \, a_{11}a_{01}=1,$$   and $J$ is not integrable.
\end{itemize}
\end{proof}

\section{$K$-Invariant complex structures on intermediate flags}

In this section we study existence of invariant almost complex structures on intermediate flags $\F_\Theta$, and their integrability. We obtain the classification of the flags admitting $K$-invariant complex structures, only some of type $C_l$ do, and also we describe the complex structures explicitly.

Proposition \ref{corminvar} states that if $\F_\Theta=K/K_\Theta$ with $\Theta\neq \emptyset$ admits a $K$-invariant almost complex structure, then $\F_\Theta$ is one of the following:
\begin{itemize}
\item of type $B_3$ and $\Theta=\{\lambda_1-\lambda_2, \lambda_2-\lambda_3\}$;
\item of type $C_l$ with $l=4$ and $\Theta=\{\lambda_1-\lambda_2,\lambda_3-\l_4\}$ or $\Theta=\{\lambda_3-\lambda_4,2\lambda_4\}$; or $l\neq 4$ and $\Theta=\{\lambda_{d}-\lambda_{d+1},\ldots,\lambda_{l-1}-\lambda_l,2\lambda_l\}$ for $d>1$, $d$ odd.
\item of type $D_l$ with $l=4$ and $\Theta$ being one of: $\{ \lambda_{1}-\lambda_{2},\lambda_{3}-\lambda_{4}\}$, $\{ \lambda_{1}-\lambda_{2},\lambda_{3}+\lambda_{4}\}$, $\{ \lambda_{3}-\lambda_{4},\lambda_{3}+\lambda_{4}\}$, $\{\lambda_1-\lambda_{2}, \lambda_{2}-\lambda_{3},\lambda_{3}-\lambda_{4}\}$, $\{\lambda_1-\lambda_{2}, \lambda_{2}-\lambda_{3},\lambda_{3}+\lambda_{4}\}$, $\{\lambda_2-\lambda_{3}, \lambda_{3}-\lambda_{4},\lambda_{3}+\lambda_{4}\}$; or $l\geq 5$ and $\Theta=\{\lambda_d-\lambda_{d+1},\cdots, \lambda_{l-1}-\lambda_{l},\lambda_{l-1}+\lambda_{l}\}$ for $1<d\leq l-1$.
\end{itemize}

We analyse the cases $B$, $C$ and $D$ separately in the next subsections. We need to treat them separately since the isotropy representations differ significantly. Nevertheless the techniques applied follow the request of necessary conditions, we shall describe below.

Recall that $K$-invariant almost complex structures on $\F_\Theta$ are in one to one correspondence with $K_\Theta$-invariant maps $J:\mn_\Theta^-\lra \mn_\Theta^-$ such that $J^2=-1$. 

Assume $J:\mn_\Theta^-\lra \mn_\Theta^-$ is $K_\Theta$-invariant and $J^2=-1$. Then $J$ is necessarily $M$-invariant since $M\subset K_\Theta=M(K_\Theta)_0$, hence by Proposition \ref{propminvar} we have 
\begin{equation}
JV_{[\alpha]}=V_{[\alpha]}\mbox{ for each }\alpha\in \Pi^-\bs \lela\Theta\rira^-. \label{Valphainv}
\end{equation} 
In addition, $J$ is also $(K_\Theta)_0$ invariant and therefore \begin{eqnarray}
&\ad_X J=J\ad_X \mbox{ for all }X\in \mk_\Theta.&\label{kthetainv}
\end{eqnarray} 

Assume $\mn_\Theta^-=W_1\oplus \cdots\oplus W_s$ is a decomposition on  $K_\Theta$-invariant and irreducible subspaces. If the representation on $W_i$ is not equivalent to the representation on any other $W_j$, $j\neq i$ then $JW_i=W_i$ because of Lemma \ref{lem2}. Notice that if this is the case $W_i$ is even dimensional. To the contrary, if $JW_i=W_j$ for some $i\neq j$, then the $K_\Theta$ representation on these subspaces are equivalent, and $J$ gives such a equivalence.

To address the non-existence of almost complex structures, we prove that some of necessary conditions above cannot hold simultaneously. For the cases where an invariant almost complex structure does exists, we use these necessary conditions to build them explicitly. Notice that, for instance, if $J:\mn_\Theta^-\lra \mn_\Theta^-$ with $J^2=-1$ satisfying \eqref{Valphainv} and \eqref{kthetainv} is $K_\Theta$ invariant.

We remark that the conditions related to the $K_\Theta$ and $\mk_\Theta$ representation on $\mn_\Theta^-$ are dealt through a description of $\mgg$ as a matrix Lie algebra. Integrability of the almost complex structure is established by computing the Nijenhuis tensor, as in the maximal flag case.

\subsection{Flags of $B_3=\so(3,4)$.}

The set of simple roots is $\Sigma=\{\lambda_1-\lambda_2,\lambda_2-\lambda_3,\lambda_3\}$, and we take $\Theta=\{\lambda_1-\lambda_2, \lambda_2-\lambda_3\}$ obtaining $\langle \Theta \rangle=\pm \{\lambda_1-\lambda_2, \lambda_2-\lambda_3,
\lambda_1-\lambda_3\}$. Notice that the flag is a  six dimensional manifold. The $M$-equivalence classes outside of $\Theta$ are: $\{\lambda_1+\lambda_2,
\lambda_3\}$, $\{\lambda_1+\lambda_3, \lambda_2\}$ and $\{\lambda_2+\lambda_3, \lambda_1\}$. The compact subgroup $\left( K_{\Theta}\right)_0$ is isomorphic to $\mbox{SO}\mbox(3)$.

We consider the realization of $B_{3}=\mathfrak{so}(3,4)$ in real
matrices of the type 
\begin{equation*}
\left( 
\begin{array}{ccc}
0 & \beta & \gamma \\ 
-\gamma ^{T} & A & B \\ 
-\beta ^{T} & C & -A^{T}%
\end{array}%
\right) ,
\end{equation*}%
with $A,B,C$ are $3\times 3$ matrices, $\beta ,\gamma $ $1\times 3$ matrices  and 
$B+B^{T}=C+C^{T}=0$. Then, $\left( K_{\Theta }\right) _{0}$ (respectively $M$%
) is given by matrices of the form
\begin{equation*}
\left( 
\begin{array}{ccc}
1 & 0 & 0 \\ 
0 & g & 0 \\ 
0 & 0 & g%
\end{array}%
\right) ,
\end{equation*}%
with $g\in \mbox{SO}\mbox(3)$ (respectively $g$ diagonal with entries $\pm 1$ and an even amount of $-1$ entries). The root space corresponding to the short root $\lambda _{1}$ is given by
matrices where the components $A,B,C$ and $\beta $  vanish and $\gamma 
$ is a multiple of $e_{1}=(1,0,0)$. The same holds for the roots $\lambda _{2}$
and $\lambda _{3}$ with $e_{2}=(0,1,0)$ and $e_{3}=(0,0,1)$, respectively.
The root spaces corresponding to $\lambda _{i}+\lambda _{j}$ have $B$ as unique non-vanishing component and it has the following form, depending on the long root:
\begin{eqnarray*}
&\lambda _{1}+\lambda _{2}:B=\left( 
\begin{array}{ccc}
0 & -1 & 0 \\ 
1 & 0 & 0 \\ 
0 & 0 & 0%
\end{array}%
\right) \ \ \ \lambda _{1}+\lambda _{3}:B=\left( 
\begin{array}{ccc}
0 & 0 & -1 \\ 
0 & 0 & 0 \\ 
1 & 0 & 0%
\end{array}%
\right)& \\
& \lambda _{2}+\lambda _{3}:B=\left( 
\begin{array}{ccc}
0 & 0 & 0 \\ 
0 & 0 & -1 \\ 
0 & 1 & 0%
\end{array}%
\right) .&
\end{eqnarray*}%

The subspaces $V_{c}=\sum_{i}\mathfrak{g}_{\lambda _{i}}$ and $V_{l}=\sum_{i,j}%
\mathfrak{g}_{\lambda _{i}+\lambda _{j}}$ are both invariant subspaces under the adjoint
representation of $K_{\Theta }=M\cdot\mbox{SO}\mbox(3)$. The representation of
the $\mathrm{SO}\left( 3\right) $ on $V_{c}$ is isomorphic to canonical
representation on $\mathbb{R}^{3}$, while the representation on $V_{l}$ is
the adjoint representation. These two representations of $\mbox{SO}\mbox(3)$
are isomorphic. In fact, an isomorphism is constructed via the identification of $%
\mathbb{R}^{3}$ with the imaginary quaternions $\mathbb{H}$: if $p,q\in 
\mathbb{H}$ then $\ad(q)p=[q,p]\in \mathrm{Im}\ \mathbb{H}$ and $\ad(q)\in \mathfrak{so}(3)$ that commutes with the representations of the $\mbox{SO}\mbox(3)$.
This isomorphism also commutes with the representations of $M$. Indeed, considering the basis $\{e_{1},e_{2},e_{3}\}=\{i,j,k\}\in \mathbb{R}^{3}=\mathrm{Im}\ \mathbb{H}$, we have
\begin{eqnarray*}
&\mathrm{ad}(i)=\left( 
\begin{array}{ccc}
0 & 0 & 0 \\ 
0 & 0 & -2 \\ 
0 & 2 & 0%
\end{array}%
\right), \; \mathrm{ad}(j)=\left( 
\begin{array}{ccc}
0 & 0 & 2 \\ 
0 & 0 & 0 \\ 
-2 & 0 & 0%
\end{array}%
\right) &\\
& \mbox{and}\;\;  \mathrm{ad}(k)=\left( 
\begin{array}{ccc}
0 & -2 & 0 \\ 
2 & 0 & 0 \\ 
0 & 0 & 0%
\end{array}%
\right) .&
\end{eqnarray*}%

The isomorphism $P: V_c\rightarrow V_l$ takes the root spaces $\mathfrak{g}_{\lambda
_{1}}$, $\mathfrak{g}_{\lambda _{2}}$ and $\mathfrak{g}_{\lambda _{3}}$ to
the root spaces $\mathfrak{g}_{\lambda _{2}+\lambda _{3}}$, $\mathfrak{g}%
_{\lambda _{1}+\lambda _{3}}$ and $\mathfrak{g}_{\lambda _{1}+\lambda _{2}}$%
, respectively. In addition, it commutes with
the representation of $(K_{\Theta})_0$ and with the representations of $M$.
Therefore, $P: V_c\rightarrow V_l$ commutes
with the representation of $K_{\Theta}$.

\begin{proposition} The flag manifold $\F_\Theta$ of $B_3$ with $\Theta=\{\l_1-\l_2,\l_2-\l_3\}$ admits $K$-invariant almost complex structures and each of them is given by $J_a$ for some $a \neq 0$ where $ J_a:\mn_\Theta^+\lra \mn_\Theta^+$ is defined by
$$ J_a(X)=aP(X) \mbox{ if }X \in V_c,\quad J_a(X)=-aP^{-1}(X) \mbox{ if }X\in V_l.$$
 These structures are not integrable.
\end{proposition}
\begin{proof}
We have $\mathfrak{n}_{\Theta}^+=V_c\oplus V_l$ as $K_\Theta$-invariant irreducible subspaces and because of the reasoning above, $J_a$  is indeed invariant by $K_{\Theta}$. Thus, there is a one-parameter family of invariant almost complex structures on $\mathbb{F}_{\Theta}$.

Furthermore, a $K_\Theta$-invariant complex structure $J$ on $\mn_\Theta^+$ is of this form. 
In fact, any $K_\Theta$-invariant complex structure $J:\mn_\Theta^+\lra \mn_\Theta^+$ interchanges $V_c$ with $V_l$ by \ref{Wirred}. in Lemma \ref{lem1}, since these are irreducible  odd dimensional subspaces. Moreover the subspaces $\mgg_{\l_1+\l_2}\oplus \mgg_{\l_3}$, $\mgg_{\l_1+\l_3}\oplus \mgg_{\l_2}$, $\mgg_{\l_2+\l_3}\oplus \mgg_{\l_1}$ are $J$-invariant because of \eqref{Valphainv}. 
The fact that $\ad_XJ=J\ad_X$ for all $X\in \mk_\Theta$ implies that $J$ is actually a multiple of $P$.

These structures are never integrable. In fact, $[V_c,V_c]= V_l$ and $[V_l,\mn_\Theta^+]=0$. Thus, for $X,Y\in V_c$ we have $J_aX, J_aY \in V_l$ and therefore $N_{J_a}(X,Y)=-[X,Y]$. Hence $N_J$ never vanishes.
\end{proof}

\begin{observation}
This flag $\F_\Theta$ of type $B_3$ and $\Theta=\{\lambda_1-\lambda_2, \lambda_2-\lambda_3\}$ is the Grassmannian of three dimensional isotropic
subspaces of $\mathbb{R}^{7}$, that is, three dimension subspaces in
which the quadratic form matrix%
\begin{equation*}
\left( 
\begin{array}{ccc}
1 & 0 & 0 \\ 
0 & 0 & 1_{3\times 3} \\ 
0 & 1_{3\times 3} & 0%
\end{array}%
\right)
\end{equation*}%
vanishes. The proposition above gives a family of $K$-invariant almost complex structures on this flag which is
parametrized by $\mathbb{R}\backslash \{0\}$.
\end{observation}

\subsection{Flags of $C_l=\mathfrak{sp}(l,\mathbb{R})$}

The set of simple roots
is $\Sigma=\{\lambda_1-\lambda_2,\ldots,\lambda_{l-1}-\lambda_l,2\lambda_l\}$. For the analysis of these flags, we separate the case $l=4$ where the $M$-equivalence classes are different from the general case.

\subsubsection{Case $C_l$, $l\neq 4$}
Assume $l\neq 4$ and let $\Theta=\{\lambda_{d+1}-\lambda_{d+2},\ldots,\lambda_{l-1}-\lambda_{l},2\lambda_l\}$ with $d\in\{0,\cdots,l\}$ and $d$ even. Notice that $\Theta$ gives a Dynkin sub diagram $C_p$ of $C_l$ with $p=l-d$, thus $\mathfrak{k}_{\Theta}$ is the maximal compact subalgebra of $\mathfrak{sp}(p,\mathbb{R})$, that is, $\mathfrak{k}
_{\Theta}\simeq \mathfrak{u}(p)$.

The $M$-equivalence classes in $\Pi^+\bs\lela\Theta\rira^+$ are $$\{\lambda_i-\lambda_{j},\lambda_i+\lambda_{j}\},\, 1\leq i\leq d,\ i<j\leq l,\; \mbox{ and }\;\{2\lambda_{1}, \ldots, 2\lambda_{d}\}.$$ For each positive root $\alpha$ denote $\mt_\alpha=(\mathfrak{g}_{\alpha}\oplus \mathfrak{g}_{-\alpha})\cap \mathfrak{k}$. Then  $\mk=\mk_\Theta\oplus \mm_\Theta$ where
$\mk_\Theta$ is the vector space sum of $\mt_\alpha$ where $\alpha$ runs in $ \lela\Theta\rira^+$ and $$\mm_\Theta=\sum_{1\leq i\leq d, i<j\leq l}\mt_{\lambda_i-\lambda_j} \oplus \mt_{2\lambda_1}\oplus\cdots\oplus \mt_{2\lambda_d}  $$
is a reductive complement of $\mk_\Theta$.

The invariant and
irreducible subspaces of $\mm_\Theta$ by the $K_{\Theta}$ action were described in \cite[Section 5.3]{patrao2015isotropy} and we present them below. Define
\begin{eqnarray*}
R&=&\{\lambda_i\pm\lambda_j: 1\leq i<j\leq d\}\cup \{2\lambda_i: 1\leq i\leq d \}.\\
\Pi_i&=&\{\lambda_i\pm\lambda_j: d+1\leq j\leq l\},\quad i=1\ldots,d, 
\end{eqnarray*}
and let $W_R=\sum_{\alpha \in R}\mathfrak{k}_{\alpha}$ and $W_i=\sum_{\alpha \in \Pi_i}\mathfrak{k}_{\alpha}$, $i=1,\ldots,d$. We have 
\begin{equation}\label{mthetadescminv}
\mathfrak{m}_{\Theta}=W_R\oplus \sum_{i=1}^{d}W_i
\end{equation}
and  the subspaces above are $M$-invariant. 

If $\alpha \in R$ and $\beta \in \Theta$, then $\pm\alpha \pm \beta$ is never a root so $[Y,X]=0$ for any $Y\in \mathfrak{k}%
_{\Theta}$ and $X\in W_R$. Thus $\mbox{Ad}\mbox(g)X=X$ for any $g \in (K_{\Theta})_0$, since $(K_{\Theta})_0$ is connected, and therefore $W_R$ is invariant by $\mbox{Ad}\mbox(%
K_{\Theta})$.

Each subspace $W_i$ is $K_\Theta$ invariant and irreducible subspace and the respective representations not equivalent if $i\neq j$ (see \cite[Lemma 5.11]{patrao2015isotropy}).  
We make use of the following isomorphism between the compact algebra $\mathfrak{k}$ and $\mmu(l)$ given by
\begin{equation*}
\left( 
\begin{array}{cc}
A & -B \\ 
B & A%
\end{array}
\right)\longmapsto A+iB, \quad A+A^T=B-B^T=0.\label{isomk}
\end{equation*} The isomorphism takes $\mathfrak{k}_{\Theta}$ in the algebra of anti-hermitian matrices of the form 
\begin{equation}
\mathfrak{k}_{\Theta}: \left( 
\begin{array}{cc}
0 & 0 \\ 
0 & X%
\end{array}
\right),\label{kthetaform}
\end{equation}
being $X$ a $p \times p$ matrix. Moreover $W_R$ corresponds to the matrices of the form 
\begin{equation*}
W_R: \left( 
\begin{array}{cc}
* & 0 \\ 
0 & 0%
\end{array}
\right),
\end{equation*}
whith $d \times d$ upper left block, while the subspace $%
W=\sum_{i=1}^{d}W_i$ corresponds to 
\begin{equation}\label{Wform}
W: \left( 
\begin{array}{cc}
0 & -\overline{C}^T \\ 
C & 0%
\end{array}
\right),
\end{equation}
where $C$ is $d\times p$. A subspace $W_j$ is
given by those matrices $C$ having non vanishing entries in column $j$. The representation of $\mathfrak{k}_{\Theta}$ in $W$ is
given by the adjoint action:
\begin{equation*}
\left[\left( 
\begin{array}{cc}
0 & 0 \\ 
0 & X%
\end{array}
\right), \left( 
\begin{array}{cc}
0 & -\overline{C}^T \\ 
C & 0%
\end{array}
\right)\right]= \left( 
\begin{array}{cc}
0 & \overline{C}^TX \\ 
XC & 0%
\end{array}
\right).
\end{equation*}
Thus $C$ having non-vanishing entries on column $j$ implies the same occurs or $XC$. So the subspaces $W_j$ are, in fact, invariant.

The image of $\mathfrak{k}_{\lambda_j-\lambda_k}$ in $\mmu(l)$ through the isomorphism is generated by the real anti-symmetric matrix $A_{jk}=E_{jk}-E_{kj}$,
while the image of  $\mathfrak{k}_{\lambda_j+\lambda_k}$ is generated by the
imaginary symmetric matrix $S_{jk}=i(E_{jk}+E_{kj})$.

\begin{lemma}\label{aacc} \label{aacd} 
\begin{enumerate}
\item An almost complex structure $J: \mm_\Theta\lra \mm_\Theta$  is $M$-invariant if and only if $J$ leaves invariant each subspace $\mathfrak{k}_{\lambda_i-\lambda_j}\oplus \mathfrak{k}_{\lambda_i+\lambda_j}$ and \linebreak $\mathfrak{k}_{2\lambda_1}\oplus \cdots\oplus 
\mathfrak{k}_{2\lambda_{d}}$.

\item An $M$-invariant almost complex structure $J$ is $K_{\Theta}$-invariant if and only if for each $j=1,\ldots,d$ there is some  $\varepsilon_{j}=\pm 1$ such that $JA_{kj}=\varepsilon_{j}S_{kj}$
and $JS_{kj}=-\varepsilon_{j}A_{kj}$ for all $d<k\leq l$.\end{enumerate}
\end{lemma}

\begin{proof} Let  $J: \mm_\Theta\lra \mm_\Theta$  be an isomorphism such that $J^2=-1$.  From Proposition \ref{propminvar} and taking into account the $M$-equivalence classes given above we have that $J$ is $M$-invariant if and only if it preserves each $\mathfrak{k}_{\lambda_i-\lambda_j}\oplus \mathfrak{k}_{\lambda_i+\lambda_j}$ and $\mathfrak{k}_{2\lambda_1}\oplus \cdots\oplus 
\mathfrak{k}_{2\lambda_{d}}$.

Now assume $J$ is $M$-invariant, then $J$ is  $K_\Theta$-invariant if and only if $\ad_YJ=J\ad_Y$ for all $X\in \mk_\Theta$.

Notice that $J$ preserves each $W_i$ and $W_R$ in \eqref{mthetadescminv}. Since $[X,Y]=0$ for all  $Y\in \mk_\Theta$, $X\in W_R$ we see that $J|_{W_R}$ is $K_\Theta$-invariant. Recall that $W_i$ is spanned by $A_{ji}$, $S_{ji}$ with $d+1\leq j\leq l$.

 Let $Y\in \mk_\Theta$ be as in \eqref{kthetaform} with $X$ imaginary diagonal matrix, i.e., $X=\mathrm{diag}(ia_1,\ldots,ia_m)$. We have $\mbox{ad}\mbox(Y)A_{kj}=a_jS_{kj}$ and $\mbox{ad}\mbox(%
Y)S_{kj}=-a_jA_{kj}$. That is, $\mathfrak{k}_{\lambda_j-\lambda_k}\oplus 
\mathfrak{k}_{\lambda_j+\lambda_k}$ is invariant by $\mbox{ad}\mbox(Y)$ and
the matrix of $\mbox{ad}\mbox(Y)$ in the basis $\{A_{kj},S_{kj}\}$ is 
\begin{equation}\label{adYkj}
\left( 
\begin{array}{cc}
0 & -a_j \\ 
a_j & 0%
\end{array}
\right).
\end{equation}

If we denote $J_{kj}$ the restriction of $J$ to $\mathfrak{k}_{\lambda_j-\lambda_k}\oplus \mathfrak{k}_{\lambda_j+\lambda_k}$, for $k>j$ we see that $J_{kj}$ commutes with $\ad(Y)$ only when its matrix in the basis $\{A_{kj},S_{kj}\}$ is
\begin{equation}
J_{kj}=\varepsilon_{kj}\left( 
\begin{array}{cc}
0 & -1 \\ 
1 & 0%
\end{array}
\right) \quad \mbox{ with }\varepsilon_{kj}=\pm1.
\label{Jkjform}
\end{equation}

Fix $j\in \{1,\ldots,d\}$ and let $l\geq s,t\geq d+1$, consider $Z$ be as in Eq. \eqref{kthetaform} with $X=E_{ts}-E_{st}$ and let $D$ be as in Eq. \eqref{Wform} with $C=E_{sj}$. Then
\begin{equation*}
\mbox{ad}\mbox(Z)D=\left( 
\begin{array}{cc}
0 & -\overline{XC}^T \\ 
XC & 0%
\end{array}
\right), \qquad \mbox{ with }XC=E_{tj}.
\end{equation*}
This implies that that $\mbox{ad}\mbox(Z)A_{sj}=A_{tj}$ and $\mbox{ad}\mbox(Z)S_{sj}=S_{tj}$. Recall that $J$ in the basis restricted to $\mk_{\l_j-\l_k}\oplus \mk_{\l_j+\l_k}$ has a matrix of the form in Eq. \eqref{Jkjform} in the appropriate basis. In order $J$ to commute with $\ad(Z)$ above,  we need
$$\varepsilon_{tj}S_{tj}=JA_{tj}=J \ad(Z) A_{sj}= \ad(Z)J A_{sj}=\ad(Z)\varepsilon_{sj}S_{sj}=\varepsilon_{sj}S_{tj}. $$ Thus $\varepsilon_{sj}=\varepsilon_{tj}$ for all $l\geq s,t\geq d+1$, and we define $\varepsilon_j$ this value. We have then $JA_{kj}=\varepsilon_{j}S_{kj}$ and $JS_{kj}=-\varepsilon_{j}A_{kj}$ for all $d<k\leq l$.

Next we prove that this condition is sufficient for $J$ to commute with  the adjoint of elements in $\mk_\Theta$. Indeed, for $j,s,t$ as above, we only have left to verify that $J$ commutes with matrices $Z$ as in Eq. \eqref{kthetaform} with with $X=i(E_{ts}+E_{st})$. We consider $D$ as in Eq. \eqref{Wform} with $C=E_{sj}$, then $XC=iE_{tj}$ and we obtain $\mbox{ad}\mbox(Z)A_{sj}=S_{tj}$. Likewise, if $C=iE_{sj}$, then $XC=-E_{tj}$ and
thus  $\mbox{ad}\mbox(Z)S_{sj}=-A_{tj}$. Therefore 
\begin{equation*}
\mbox{ad}\mbox(Z)JA_{sj}=\varepsilon_{j}\mbox{ad}\mbox(Z)S_{sj}=-%
\varepsilon_{j}A_{tj}=JS_{tj}=J\mbox{ad}\mbox(Z)A_{sj}
\end{equation*}
and 
\begin{equation*}
\mbox{ad}\mbox(Z)JS_{sj}=-\varepsilon_{j}\mbox{ad}\mbox(Z)A_{sj}=-%
\varepsilon_{j}S_{tj}=-JA_{tj}=J\mbox{ad}\mbox(Z)S_{sj}.
\end{equation*}
\end{proof}

\begin{observation}
The set of $K$ invariant almost complex
structures in the  flags $\F_\Theta$ of the proposition is parametrized by
$\mbox{Gl}\mbox(d-1,\mathbb{R})/ \mbox{Gl}\mbox(d-1/2,\mathbb{C})\times (%
\mathbb{R}^2\cup \mathbb{R}^2 )^{d(d-1)}\times \mathbb{Z}_2^{d}.$

The component $\mbox{Gl}\mbox(d-1,\mathbb{R})/ \mbox{Gl}\mbox(d-1/2,\mathbb{C%
})$ corresponds to the complex structures on the space generated by long roots outside $\langle \Theta\rangle^+$. The component $(\mathbb{R}^2\cup \mathbb{R}^2
)^{d(d-1)}$ corresponds to the structures on the spaces generated by the roots $\{\lambda_j-\lambda_k,\lambda_j+\lambda_k\}$.
The set $\mathbb{R}^2\cup \mathbb{R}^2 $ is the disjoint union of the two copies
of $\mathbb{R}^2$, that is $\mbox{Gl}\mbox(2,\mathbb{R})/ \mbox{Gl}\mbox(1,%
\mathbb{C})$. Finally, $\mathbb{Z}_2^{(d-1)}$ parametrizes the signs $\varepsilon_{j}$.
\end{observation}
\smallskip

We introduce  two  technical lemmas which will lead to the determination of the integrable structures.

\begin{lemma}\label{lema2} Let $J$ be a $K_\Theta$-invariant almost complex structure. If $J$ is integrable then for each $i,j\in\{1,\ldots,d\}$, $j>i$, we have $JA_{ji}=c_{ji}S_{ji}$ and $JS_{ji}=-c_{ji}A_{ji}$, with $c_{ji}=\pm 1$.
\end{lemma}
\begin{proof}
Take $1\leq i<j\leq d$ then by $M$-invariance 
$JS_{ii}=\sum_k b_{ki}S_{kk}$ and 
$$J|_{\{A_{ji},S_{ji}\}}=
\left(\begin{matrix}
a_{ji}&-\frac{1+a_{ji}^2}{c_{ji}}\\
c_{ji}&-a_{ji}
\end{matrix}\right)\mbox{ where } c_{ji}\neq 0.
$$
We have
\begin{eqnarray*}
&N_J(S_{ii},A_{ji})
=A_{ji}\left(2c_{ji}(b_{ii}-b{ji})+2(b_{ji}-b_{ii})\frac{(1+a_{ji}^2)}{c_{ji}}-2c_{ji}a_{ji}-2a_{ji}\frac{(1+a_{ji}^2)}{c_{ji}}\right)&\\
&\qquad\qquad +S_{ji}\left(
2a_{ji}(b_{ji}-b_{ii})+2+2a_{ji(b_ji}-b_{ii})-2(a_{ji}^2+c_{ji} ^2)\right)&
\end{eqnarray*}
Therefore
\begin{eqnarray*}
N_J(S_{ii},A_{ji})=0 
&\Leftrightarrow &\left\{
\begin{array}{l}
a_{ji}=0\\
c_{ji}=\pm 1\end{array}\right..
\end{eqnarray*}
\end{proof}

Up to this moment we have proved that if $J$ is $K_\Theta$-invariant and integrable then for each $j=1,\ldots,d$:
\begin{itemize}
\item $JA_{kj}=c_{kj}S_{kj}$ and $JS_{kj}=-c_{kj}A_{kj}$ for $k=1,\ldots,d$, $k\neq j$ and 
\item $JA_{kj}=\varepsilon_{j}S_{kj}$
and $JS_{kj}=-\varepsilon_{j}A_{kj}$ for all $k=d+1,\ldots, l$.
\end{itemize} where $\varepsilon_j,c_{kj}\in\{\pm 1\}$. To simplify notation in the following lemma we write
\begin{equation}\label{eq.Es}
JA_{kj}=\mu_{kj}S_{kj}, \quad JS_{kj}=-\mu_{kj}A_{kj} \mbox{for all } j=1,\ldots,d, \,j< k\neq l.
\end{equation}

\begin{lemma}
\label{propes} Let $J$ be a $K_\Theta$-invariant (integrable) complex structure. Then for any triple $k>j>s$ such that  $j,s\in\{1,\ldots,d\}$ the possible
values for $(\mu_{ks},\mu_{kj},\mu_{js})$ are:
$$(\mu_{ks},\mu_{ks},\mu_{ks}), \;(\mu_{ks},-\mu_{ks},\mu_{ks})\,\mbox{ and }(\mu_{ks},\mu_{ks},-\mu_{ks})., \quad \mu_{ks}=\pm1.$$
In particular, if $\varepsilon_j=-\varepsilon_s$  then $c_{js}=\varepsilon_s$.
\end{lemma}
\begin{proof}
By equation \eqref{eq.Es} we obtain
\begin{eqnarray*}
0\,=\,N_J(A_{kj},A_{ks})&=&\left(1+\mu_{kj}\mu_{js}-
\mu_{kj}\mu_{ks}-\mu_{ks}\mu_{js}\right)A_{js}\\
&=&\left((\mu_{kj}-\mu_{ks})\mu_{js}+(
\mu_{ks}-\mu_{kj})\mu_{ks}\right)A_{js}\\
&=&\left((\mu_{js}-\mu_{ks})\mu_{kj}+(
\mu_{js}-\mu_{ks})\mu_{js}\right)A_{js}.
\end{eqnarray*}
From the second row of this equation we see that  $\mu_{kj}=-\mu_{ks}$ implies $\mu_{js}=\mu_{ks}$; while the third row implies $\mu_{kj}=-\mu_{js}=\mu_{ks}$ if  $\mu_{js}=-\mu_{ks}$.
We conclude then that the possible values for the triple $(\mu_{ks},\mu_{kj},\mu_{js})$ are:
$(\mu_{ks},\mu_{ks},\mu_{ks})$, $(\mu_{ks},-\mu_{ks},\mu_{ks})$ and $(\mu_{ks},\mu_{ks},-\mu_{ks})$.
\end{proof}

\begin{proposition} Let $J:\mm_\Theta\lra \mm_\Theta$ be such that $J^2=-1$ and moreover it preserves $\mk_{2\l_1}\oplus\cdots\oplus\mk_{2l_d}$ and  $JA_{kj}=\mu_{kj}S_{kj}$, $JS_{kj}=-\mu_{kj}A_{kj}$ for all $j=1,\ldots,d$, $j<k\leq d$, with $\mu_{kj}=\pm 1$.

Then $J$ is $K_\Theta$-invariant and integrable if and only if the following hold:
\begin{itemize}
\item for each $j=1,\ldots,d$, $\mu_{kj}=\varepsilon_j$ for all $k=d+1,\ldots,l$.
\item for each triple $k>j>s$ such that $j,s\in\{1,\ldots,d\}$ the coefficients $(\mu_{ks},\mu_{kj},\mu_{js})$ are one of the following:
$$(\mu_{ks},\mu_{ks},\mu_{ks}), \;(\mu_{ks},-\mu_{ks},\mu_{ks})\,\mbox{ and }(\mu_{ks},\mu_{ks},-\mu_{ks}).$$

Conversely, any $K$-invariant complex structure on $\F_\Theta$ is induced by $J$ as above.
\end{itemize}
\end{proposition}
\begin{proof} It is necessary for $J$ to be $M$-invariant to  preserve $\mk_{2\l_1}\oplus\cdots\oplus\mk_{2l_d}$ and $\mk_{\l_j-\l_k}\oplus \mk_{\l_j+\l_k}$.  
The conditions above are necessary as proved in Lemma \ref{aacd} in order $J$ to be $K_\Theta$-invariant and Lemma \ref{lema2} and Lemma \ref{propes} to be integrable. As seen there, such $J$ verifies
$N_J(S_{kk},A_{kj})=0$ $j=1,\ldots,d$, $j<k\leq l$ and $N_J(A_{kj},A_{ks})=0$ for each triple in the second item. To show that these conditions are sufficient we have to show that 
i) $N_J(S_{kk},S_{kj})=0$, ii) $N_J(S_{kj},S_{ks})=0$, iii) $N_J(S_{kj},A_{ks})=0$ and iv) $N_J(S_{jj},S_{ss})=0$ for all $j>s\in\{1,\ldots,d\}$ and $k>j>s$.

Clearly iv) holds since these matrices are diagonal. Moreover, $N_J(A_{kj},A_{ks})=N_J(S_{kj},S_{ks})$ so ii) also holds. Similar computations as in the proof of Lemma \ref{lema2} give i). Finally $N_J(S_{kj},A_{ks})=\left(-1-\mu_{kj}\mu_{js}+
\mu_{kj}\mu_{ks}+\mu_{ks}\mu_{js}\right)S_{js}$ so reasoning as in Lemma \ref{propes} one obtains iii).
\end{proof}

\begin{example}
We consider the flag $\mathbb{F}_{\Theta}$ of $C_3$, with $\Theta=\{2\lambda_3\} $. 
The component $W_R$ of tangent space at the origin of
flag is given by sum of $\mathfrak{k}_{\alpha}$, $\alpha\in R$, and has the
following form: $R=\{\lambda_1\pm\lambda_2\} \cup \{2\lambda_1,2\lambda_2\}$%
. The components $W_j$ are determined by the sets of roots
\begin{equation*}
\Pi_1 =  \{\lambda_1\pm\lambda_3\},\quad \Pi_2 =  \{\lambda_2\pm\lambda_3\}.
\end{equation*}

Fix $\varepsilon_j=\pm1$ $j=1,2$ $\nu=\pm1$ such that
$$(\varepsilon_1,\varepsilon_2,\nu)\in\{(1,1,1),(-1,-1,-1),(1,-1,1),(-1,1,-1),(1,1,-1),(-1,-1,1)\},$$ and let $a_{11},c_{11}\in \R$ s.t. $c_{11}\neq 0$. The following table gives all $K_\Theta$-invariant integrable complex structures $J$ in  $\F_\Theta$.
\begin{table}[h]
\label{tab3}
\centering
\begin{tabular}{|c|c|}
\hline
Components & $K_{\Theta}$-invariant complex structures  \\ 
\hline\hline
$W_1$ & $JA_{31}=\varepsilon_{1}S_{31}$, $JS_{1}=-\varepsilon_{1}A_{31}$, \\ \hline
$W_2$ & $JA_{32}=\varepsilon_{2}S_{32}$, $JS_{32}=-\varepsilon_{2}A_{32}$  \\ \hline
 & $JA_{21}=\nu S_{21}$, $JS_{21}=-\nu A_{21}$, \\
$W_R$&$JS_{11}=a_{11}S_{11}+c_{11}S_{22}$,\\
& $JS_{22}=-\frac{1+a_{11}^2}{c_{11}} S_{11}-a_{11}S_{22}$ 
\\ \hline
\end{tabular}%
\end{table}
\end{example}

\subsubsection{Case $C_4$}  The $M$-equivalence classes of positive roots are 
$$\{\lambda_1\pm\lambda_2,\lambda_3\pm\lambda_4\}, \;\{\lambda_1\pm\lambda_3,\lambda_2\pm\lambda_4\}
\{\lambda_1\pm\lambda_4,\lambda_2\pm\lambda_3\},\;\{2\lambda_1, 2\lambda_2, 2\lambda_3, 2\lambda_4\}.$$ 

\begin{proposition} The real flag $\F_\Theta$ of $C_4$ with $\Theta= \{\lambda_1-\lambda_2, \lambda_3-\lambda_4\}$ does not admit $K$-invariant almost complex structures.
\end{proposition}
\begin{proof} According to \cite[Section 5.3]{patrao2015isotropy} the $K_\Theta$ irreducible components of $\mn_\Theta^-$ are given by
\begin{eqnarray*}
&\begin{array}{rclrcl}
V_1 & = & \langle X_{2\lambda_1}-X_{2\lambda_2}, X_{-\lambda_2-\lambda_1}\rangle &V_4&=& \langle X_{2\lambda_3}+X_{2\lambda_4} \rangle,\\
V_2&=&\langle X_{2\lambda_1}+X_{2\lambda_2} \rangle,
& V_3 & = & \langle X_{2\lambda_3}-X_{2\lambda_4},X_{-\lambda_4-\lambda_3}\rangle\end{array}&\\
&\begin{array}{rcl}
V_5 & = & \langle X_{\lambda_3-\lambda_1}+X_{\lambda_4-\lambda_2},X_{\lambda_3-\l_2}-X_{\l_4-\l_1}
\rangle\\
V_6&= & \langle X_{\l_3-\l_2}+X_{\l_4-\l_1}, X_{\l_4-\l_2}-X_{\l_3-\l_1} \rangle, \\ 
V_7 & = & \langle X_{-\l_3-\l_1}+X_{-\l_4-\l_2},X_{-\l_3-\l_2}-X_{-\l_4-\l_1}
\rangle \\
V_8&= & \langle X_{-\l_3-\l_2}+X_{-\l_4-\l_1}, X_{-\l_4-\l_2}-X_{-\l_3-\l_1} \rangle, \end{array}
\end{eqnarray*}
where $X_{\alpha}$ is a generator of root space $\mathfrak{g}_{\alpha}$.

The components $V_2$, $V_5$ and $V_6$ are equivalent to the components $V_4$, $V_7$ and $V_8$, respectively. The subspaces $V_1$ and $V_3$ are neither equivalent between them nor to any other representation subspace. 

Assume $J$ is a $K_\Theta$-invariant complex structure $J$ on $\mn_\Theta^-$. Then $JV_1=V_1$ since it is irreducible and non-equivalent no any other representation subspace. 
Moreover, $V_{[-\l_2-\l_1]}=\mgg_{-\l_2-\l_1}\oplus\mgg_{-\l_4-\l_3}$ and $J$ preserves this subspaces too because of its $M$-invariance. Therefore $V_1\cap V_{[-\l_2-\l_1]}=\lela X_{-\lambda_2-\l_1}\rira$ is an invariant subspace of $J$, which is a contradiction. So we conclude that no $K$-invariant complex structure exists in this case.\end{proof}

Fix $\Theta=\{\lambda_3-\lambda_4,2\lambda_4\}$ for $C_4$. The $K_{\Theta}$-irreducible components of $\mm_\Theta$ are  (\cite[Section 5.3]{patrao2015isotropy}): 
\begin{eqnarray}
&
\begin{array}{rclrcl}
V_1 & = & \mathfrak{g}_{-2\lambda_1},&\quad V_3 & = & \mathfrak{g}_{\lambda_2-\lambda_1},\\ 
V_2 & = & \mathfrak{g}_{-2\lambda_2}, & \quad V_4 & = & \mathfrak{g}_{-\lambda_2-\lambda_1},\end{array}&\nonumber\\
&\begin{array}{rcl}
V_5 & = & \mathfrak{g}_{\lambda_3-\lambda_1}%
\oplus \mathfrak{g}_{-\lambda_3-\lambda_1}\oplus \mathfrak{g}%
_{\lambda_4-\lambda_1}\oplus \mathfrak{g}_{-\lambda_4-\lambda_1}, \\ 
V_6 & = & \mathfrak{g}_{\lambda_3-\lambda_2}%
\oplus \mathfrak{g}_{-\lambda_3-\lambda_2}\oplus \mathfrak{g}%
_{\lambda_4-\lambda_2}\oplus \mathfrak{g}_{-\lambda_4-\lambda_2}.
\end{array}\label{descKinvC4}
\end{eqnarray}
The components $V_1$ and $V_3$ are equivalent to, respectively, the components $V_2$ and $V_4$. The components $V_5$ and $V_6$ are not equivalent. 

As in the previous section, we consider the isomorphism between $\mk$ and $\mmu(4)$. Under this map, $\mk_\Theta=\lela \{A_{43},S_{43},S_{33},S_{44}\}\rira$ and 
$$\mm_\Theta=W_R\oplus \bigoplus_{
\begin{smallmatrix}j=1,2\\k=3,4\end{smallmatrix}} W_{kj}  $$
where $W_R=W_R^1\oplus W_{21}$ with
$W_R^1=\lela \{S_{11},S_{22}\}\rira$ and $ W_{kj}=\lela\{A_{kj},S_{kj}\}\rira$.

\begin{proposition} The flag manifold $\F_\Theta$ of $C_4$ with $\Theta=\{\l_3-\l_4,2\l_4\}$ admits $K$-invariant almost complex structures and each of them is induced by a map $J:\mm_\Theta\lra \mm_\Theta$ verifying
$$\begin{array}{rclrcll}
JS_{11}&=&\nu_1 S_{22}, &JS_{22}&=&-\nu_1^{-1}S_{11}&\mbox{ with }\nu_1\neq 0,\\
JA_{21}&=&\nu_2 S_{21}, & JS_{21}&=&-\nu_2^{-1} A_{21} &\mbox{ with } \nu_2\neq 0,\\
JA_{kj}&=&\varepsilon_{j} S_{kj}, &JS_{kj}&=&-\varepsilon_{j}A_{kj} & \mbox{ for } k\in\{3,4\},\,j\in \{1,2\},\\
&&&&&& \mbox{ with }\varepsilon_{j}=\pm 1.
\end{array}$$
Such structure is integrable if and only if $\nu_2=\pm 1$ and $\nu_2=\varepsilon_1$ if $\varepsilon_2=-\varepsilon_1$.
\end{proposition}
\begin{proof}
We already know that $\F_\Theta$ admits $M$-invariant almost complex structures and such $J$ is the direct sum of almost complex structures in each $V_{[\alpha]}$,\linebreak $\alpha\in \Pi^+\bs \lela \Theta\rira^+$.  In this case, the $M$-equivalence classes  are
$$\{\lambda_1-\lambda_2,\lambda_1+\lambda_2\},\; \{\lambda_1\pm \lambda_3,\lambda_2\pm\lambda_4\},\;\{\lambda_1\pm\lambda_4,\lambda_2\pm\lambda_3\},\;\{2\lambda_1, 2\lambda_2\}.$$
So, in particular, $W_R^1$, $W_{21}$, $W_{31}\oplus W_{42}$ and $W_{32}\oplus W_{41}$ are $J$-invariant. 

Moreover, since $V_5=W_{31}\oplus W_{41}$ and $V_6=W_{32}\oplus W_{42}$ in \eqref{descKinvC4} are irreducible and non-equivalent, we have $JV_5=V_5$ and $JV_6=V_6$. Therefore each $W_{kj}$, $k=3,4$, $j=1,2$ is invariant, since it can be described as an intersection of $V_{[\alpha]}$ and $V_t$ for suitable root and index.

We proceed as in the general case $C_l$, $l \neq 4$ to show that $J$ has the form given in the statement of the proposition.

For any $Y\in \mk_\Theta$ and $Z\in W_R$, we have $[Y,Z]=0$ so $J$ restricted to this subspace is also $\mk_\Theta$-invariant. Let $Y=a_3 S_{33} +a_4 S_{44}\in \mk_\Theta$, then $\ad_Y J=J\ad_Y$ implies that for $k=3,4$, $j=1,2$ the matrix of $J|_{W_{kj}}$ in the basis $\{A_{kj},S_{kj}\}$ is 
$$\mu_{kj}\left(
\begin{array}{cc}
0&-1\\
1&0
\end{array}\right),\, \mu_{kj}=\pm 1. $$
Now let $Y=a_3 A_{43} +a_4 S_{43}\in \mk_\Theta$ and let $Z\in W_{kj}$ with $k=3,4$, then $\ad_Y JZ =J\ad_Y Z$ holds if and only if $\varepsilon_{4j}=\varepsilon_{3j}$ for $j=1,2$. It is not hard to see that these conditions are also sufficient for $J$ to be $K_\Theta$-invariant.

To address integrability, notice that, as in the general case, we have
\begin{eqnarray*}
N_J(S_{11},A_{21}) &=& -2\left(\nu_1( \nu_2-\nu_2^{-1})A_{21}+(-1+\nu_2^{2})S_{21}\right)\\
N_J(A_{41},A_{42}) &=&\left(\varepsilon_1\varepsilon_2-1+(\varepsilon_1-\varepsilon_2)\nu_2^{-1}\right)A_{21}
\end{eqnarray*}
Therefore $J$ is integrable if $\nu_2=\pm 1 $ and $\nu_2=\epsilon_1$ in the case that $\epsilon_1=\epsilon_2$. One can check that these conditions are sufficient for $J$ to be integrable.
\end{proof}

\subsection{Flags of $D_l=\mathfrak{so}(l,l)$}

A root system is given by  $\pm \lambda_i\pm\lambda_j$, $i\neq j$, and the corresponding set of simple roots is given by $\Sigma=\{\lambda_1-\lambda_2,\ldots, \lambda_{l-1}-\lambda_l,\lambda_{l-1}+\lambda_l\}$, $1\leq i<j\leq l$. The maximal compact subalgebra of $\mathfrak{so}(l,l)$ is $\mk\simeq\mathfrak{so}\left( l\right)\oplus \mathfrak{so}\left( l \right)$.

As in the $C_l$ case, we deal first with the case $D_l$ with $l\geq 5$ and later we address the case of $l=4$ because of the difference between the $M$-equivalence classes.

\subsubsection{Case $D_l$, $l\geq 5$}

We consider $\Theta =\{\lambda _{d}-\lambda _{d+1},\ldots ,\lambda
_{l-1}-\lambda_{l},\lambda _{l-1}+\lambda _{l}\}$, this gives a sub diagram $D_p$ of $D_l$ with $p=l-d+1$, thus $\mk_\Theta\simeq \mathfrak{so}\left( p\right) _{1}\oplus 
\mathfrak{so}\left( p\right) _{2}$. The set $\langle \Theta \rangle $ of roots generated by $\Theta $ is given by \begin{equation*}
\langle \Theta \rangle =\{\pm \left( \lambda _{i}\pm \lambda _{j}\right)
:d\leq i<j\leq l\}.
\end{equation*}

The roots in $\Pi^+\bs \lela \Theta\rira^+$ are 
$$
\lambda_i\pm\lambda_j\mbox{ with } 1\leq i<j\leq d, \quad \mbox{ and } \lambda_i\pm\lambda_j\mbox{ with }i=1,\ldots,d-1,\,j=d, \ldots l.
$$
and the $M$-equivalence classes are
$\{\l_i-\l_j,\l_i+\l_j\}$. Consider the subsets of roots in $\Pi^+\bs \lela \Theta\rira^+$:
\begin{eqnarray*}
R&=&\{\l_i\pm\l_j: 1\leq i<j\leq d\}\\
\Pi_i&=&\{\lambda_i\pm\l_j:d\leq j \leq l \}, \quad i=1, \ldots,d-1
\end{eqnarray*}
and let  $W_R=\sum_{\alpha\in R}\mgg_\alpha$ and $W_i=\sum_{\alpha\in \Pi_i} \mgg_{\alpha}$. Clearly we obtain
\begin{equation}\label{descntheta}
\mn_\Theta^+=W_R\oplus \sum_{i=1}^{d-1}W_i.
\end{equation}

The subspace $W_R$ is $K_\Theta$ invariant and irreducible. Moreover, each $W_i$ decomposes as $W_i=V^1_i\oplus V_i^2$, where $V_i^j$ is irreducible $K_\Theta$-invariant and the representations are not equivalent \cite{patrao2015isotropy}.  We present an explicit description of these subspaces.
\medskip

A split real form of $D_{l}$ is $\mathfrak{so}\left( l,l\right)$ and it is
represented by real matrices of the form 
\begin{equation}
\left( 
\begin{array}{cc}
A & B \\ 
C & -A^{T}%
\end{array}%
\right), \ \mbox{where}\mbox \ \ B+B^{T}=C+C^{T}=0.  \label{formatrizes}
\end{equation}
The algebra $\mathfrak{g}\left( \Theta \right) $ generated by $\mathfrak{g}%
_{\alpha }$, $\alpha \in \langle \Theta \rangle$ is given by matrices in Eq. \eqref{formatrizes} such that $A,B$ and $C$ have the form%
\begin{equation*}
\left( 
\begin{array}{cc}
0 & 0 \\ 
0 & \ast%
\end{array}%
\right) ,
\end{equation*} where the non-zero part is squared of size $p=l-d+1$. 
The Lie algebra $\mathfrak{g}\left( \Theta \right) $ is of type $%
D_{p}$, isomorphic to $\mathfrak{so}\left(p,p\right)$.

The compact part $\mathfrak{k}$ inside $\so(l,l)$ is given by the subset matrices in \eqref{formatrizes} having the form
\begin{equation*}
\left( 
\begin{array}{cc}
A & B \\ 
B & A%
\end{array}%
\right), \ \mbox{where}\mbox \ \ A+A^{T}=B+B^{T}=0.
\end{equation*}
It is well known that $\mathfrak{k}$ decomposes as a sum of two ideals, both isomorphic to $\mathfrak{so}\left( l\right)$.  The compact Lie algebra $\mk_\Theta$ lies inside $\mk$ and also inside $\mgg(\Theta)$ and consists of matrices of the form
\begin{equation}\label{kthetamatrices}
\left( 
\begin{array}{cc}
A & B \\ 
B & A%
\end{array}%
\right), \ \mbox{with }A, B \in \lela \{E_{st}-E_{ts}: d\leq s<t\leq l\}\rira.
\end{equation}
The Lie algebra $\mathfrak{k}_\Theta$ also decomposes as a sum of two ideals, both isomorphic to $\mathfrak{so}\left( p\right)$, which are 
\begin{eqnarray*}
\mathfrak{so}\left(p \right) _{1}&=&\left\{\left( 
\begin{array}{cc}
A & A \\ 
A & A%
\end{array}%
\right):A \in \lela \{E_{st}-E_{ts}: d\leq s<t\leq l\}\rira\right\},\\
\mathfrak{so}\left( p\right) _{2}&=&\left\{\left( 
\begin{array}{cc}
A & -A \\ 
-A & A%
\end{array}%
\right):A \in \lela \{E_{st}-E_{ts}: d\leq s<t\leq l\}\rira\right\}. 
\end{eqnarray*}

Fix $i\in \{1,\ldots, d-1\}$ and denote $S_{i}=\{X=(a_{st})\in \gl(l,\R):\, a_{st}=0 \mbox{ for all } (st)\notin \{(ij): \,j=d,\ldots,l\}\}$. For any $j=d, \ldots, l$ the root space $\mgg_{\l_i-\l_j}$ is represented by matrices \eqref{formatrizes} where $A=E_{ij}$, $C=B=0$; meanwhile, $\mgg_{\l_i+\l_j}$ is represented by the matrices of the above form where $B=E_{ij}-E_{ji}$, $A=C=0$.
Thus $W_i$ is given by
\begin{equation}\label{Wimatrices}
\left( 
\begin{array}{cc}
X & Y-Y^t \\ 
0 & -X^{T}%
\end{array}%
\right),\quad 
\mbox{ where }X, Y \in S_i.
\end{equation}
For $Z\in S_i$ denote 
\begin{equation}\label{Vimatrices}
X_Z=\left( 
\begin{array}{cc}
Z & Z-Z^{T} \\ 
0 & -Z^{T}%
\end{array}%
\right),\quad
Y_Z=\left( 
\begin{array}{cc}
Z& -Z+Z^{T} \\ 
0 & -Z^{T}%
\end{array}%
\right)\end{equation} and define 
$V_i^1=\{X_Z:\, Z\in S_i\}$, $V_i^2=\{Y_Z:\,Z\in S_i\}$. Clearly, $V_i^1,V_i^2\subset W_i$. Moreover, a matrix as in \eqref{Wimatrices} can be written as the sum of two matrices in \eqref{Vimatrices} by taking $Z=(X+Y)/2$, $Z'=(X-Y)/2$. Thus we obtain $W_i=V_i^1\oplus V_i ^2$ and $\dim V_i^1=\dim V_i^2=l-d+1=p=|\Theta|$.

We compute the $\mk_\Theta$ action on $V_i^1$ and $V_i^2$: let $N\in \mk_\Theta$ as in \eqref{kthetamatrices} and let $X_Z\in V_i^1$, $Y_Z\in V_i^2$ then $AZ=0=BZ$, $Z^TA=0=Z^TB$ so
$$[N,X_Z] = X_{-Z(A+B)}, \; \mbox{ and }\;
[N,Y_Z] =Y_{-Z(A-B)}.
$$
This implies that $\so(p)_2$ acts trivially on $V_i^1$ while for $N\in \so(p)_1$ the action is $[N,X_Z]=X_{-2ZA}$.
Similarly, $\so(p)_1$ acts trivially on $V_i^2$ while for $N\in \so(p)_2$ the action is $[N,X_Z]=X_{-2ZA}$. We conclude that the $\mk_\Theta$ representation on $V_i^1$ is equivalent to the $\so(p)\oplus \so(p)$ representation on $\R^p$ acting as usual with the first component and by zero with the second one. Similarly, the $\mk_\Theta$ representation on $V_i^2$ is equivalent to the $\so(p)\oplus \so(p)$ representation on $\R^p$ acting by zero with the first component and as usual with the second component.

We keep $i=1, \ldots, d-1$ fixed. Let $s,t\in \{d\ldots l\}$, $s\neq t$ and consider $N_{st}^1$ being as in \eqref{kthetamatrices} with $A=E_{st}-E_{ts}$ and $B=A$ (i.e. $N\in \so(p)_1$). Then 
\begin{equation}\label{MXbracket}
\left[N_{st}^1,X_{E_{is}}\right]=-2X_{E_{it}}\mbox{  and }[N_{st}^1,X_{E_{it}}]=2X_{E_{is}},\mbox{ while } [N_{st}^1,Y_{E_{ij}}]=0 \mbox{ for all }j.
\end{equation}
 Similarly, denote $N_{st}^2\in\so(p)_2$ being the matrix in $\mk_\Theta$ associated to $A=E_{st}-E_{ts}$ and $B=-A$, then
 \begin{equation}
\left[N_{st}^2,Y_{E_{is}}\right]=-2Y_{E_{it}}\mbox{  and }[N_{st}^2,Y_{E_{it}}]=2Y_{E_{is}},\mbox{ while } [N_{st}^2,X_{E_{ij}}]=0 \mbox{ for all }j.
\end{equation}
Having described the $\mk_\Theta$ representation on $\mn_\Theta^+$ we can state:
 
\begin{proposition}
The real flags $\F_\Theta$ of $D_l$ with $l \geq 5$ and $\Theta\neq \emptyset$ do not admit $K_\Theta$-invariant complex structures.
\end{proposition}

\begin{proof} Assume $J:\mn_\Theta^+\lra \mn_\Theta^+$ is a $K_\Theta$-invariant almost complex structure. As it is  $M$-invariant and each subspace in \eqref{descntheta} is sum of $M$-equivalence classes, we have that  $JW_R=W_R$ and $JW_i=W_i$ for all $i=1, \ldots,d-1$.



Recall that $W_i$ is not irreducible, for $i=1, \ldots, d-1$. Instead $W_i=V_i^1\oplus V_i^2$ where each of these subspace is invariant and irreducible by the $K_\Theta$ action, and the induced representations are not equivalent \cite{patrao2015isotropy}. By Lemma \ref{lem2}, we conclude that $V_i^1$, $V_i^2$ are $J$-invariant. In particular, $V_i^1$ and $V_i^2$ are even dimensional and thus $p$ is even.

Fix $i=1,\ldots,d-1$ and let $j\in\{d,\ldots,l\}$. In the notation \eqref{Vimatrices} one can see that  $\mgg_{\l_i-\l_j}\oplus \mgg_{\l_i+\l_j}=\lela \{X_{E_{ij}},Y_{E_{ij}}\}\rira$, which is a $J$-invariant subspace of $W_i$ because of the $M$-invariance of $J$. Thus $JX_{E_{ij}}=a_{ij}X_{E_{ij}}+c_{ij}Y_{E_{ij}}$ with $c_{ij}\neq 0$. For any $s\in \{d,\ldots,l\}$, $s\neq j$ we apply \eqref{MXbracket} and obtain
\begin{eqnarray*}
\ad_{N_{sj}^1} JX_{E_{ij}}&=&\ad_{N_{sj}^1}(a_{ij}X_{E_{ij}}+c_{ij} Y_{E_{ij}})=-2a_{ij}X_{E_{is}},\quad\mbox{ while } \\
J\ad_{N_{sj}^1} X_{E_{ij}}&=&J(-2X_{E_{is}})=-2(a_{is} X_{E_{is}}+c_{is}Y_{E_{is}}),
\end{eqnarray*}
but $c_{is}\neq 0$, contradicting the $K_\Theta$-invariance of $J$.
\end{proof}

\subsubsection{Case $D_4$}

Now we proceed to the study of flags of $D_4$ with $\Theta$ as in Table \ref{table.Minv}.
The $M$-equivalence classes of positive roots in $D_4$ are: 
\begin{eqnarray*}
&\{\lambda_{1}-\lambda_{2},\lambda_{1}+\lambda_{2},\lambda_{3}-\lambda_{4},%
\lambda_{3}+\lambda_{4}\},\
\{\lambda_{1}-\lambda_{3},\lambda_{1}+\lambda_{3},\lambda_{2}-\lambda_{4},%
\lambda_{2}+\lambda_{4}\}&\\
&
\{\lambda_{1}-\lambda_{4},\lambda_{1}+\lambda_{4},\lambda_{2}-\lambda_{3},%
\lambda_{2}+\lambda_{3}\}.&
\end{eqnarray*}

As in the general case we work with the split form $\so(4,4)$. In what follows we denote by $X_{ij}=E_{i,j}-E_{l+j,l+i}$ a generator of $\mathfrak{g}_{\lambda_i-\lambda_j}$ and by $Y_{ij}=E_{i,l+j}-E_{j,l+i}$ a generator $\mathfrak{g}_{\lambda_i+\lambda_j}$, where $E_{i,j}$ is the $8\times 8$ matrix with $1$ in the position $ij$ and zeroes elsewhere.

The group $M$ consists of $8\times 8$ diagonal matrices $\diag(\epsilon_1,\epsilon_2, \epsilon_3,\epsilon_4,\epsilon_1,\epsilon_2, \epsilon_3,\epsilon_4)$ where $\epsilon_i=\pm1$ and $\epsilon_1\epsilon_2 \epsilon_3\epsilon_4 =1$, that is, there is an even amount of $-1$'s in the diagonal of matrices of $M$.
\medskip

 \begin{proposition} The real flag manifold $\F_\Theta$ of type $D_4$ with $\Theta=\linebreak\{\lambda_1-\lambda_2,\lambda_3-\lambda_4\}$ admits $K_\Theta$ invariant almost complex structures. These structures are not integrable.
  \end{proposition}
\begin{proof}
The following is the decomposition of $\mn_\Theta^+$ in $K_\Theta$ invariant and irreducible subspaces 
$$\mn_\Theta^+=\mgg_{\l_1+\l_2}\oplus \mgg_{\l_3+\l_4}\oplus \sum_{i=1}^4 V_i,$$\vspace{-1cm}
where \begin{multicols}{2}
\begin{eqnarray*}
V_1&=&\langle X_{13}-X_{24},X_{14}+X_{23} \rangle \\
V_2&=& \langle X_{13}+X_{24},X_{14}-X_{23}\rangle
\end{eqnarray*}

\begin{eqnarray*}
V_3&=&\langle Y_{13}-Y_{24},Y_{14}+Y_{23} \rangle\\
V_4&=&\langle Y_{13}+Y_{24},Y_{14}-Y_{23}\rangle
\end{eqnarray*}\end{multicols}
The map $T_{13}:V_1\longrightarrow V_3$ defined by $T_{13}(X_{13}-X_{24})=Y_{13}-Y_{24}$ and $T_{13}(X_{14}+X_{23})=Y_{14}+Y_{23}$ commutes with $\ad_{\mk_\Theta}$. 
Moreover, the linear map \linebreak $T_{24}:V_2\longrightarrow
V_4$, verifying $T_{24}(X_{13}+X_{24})=Y_{23}+Y_{24}$ and $T_{24}(X_{14}-X_{23})=Y_{14}-Y_{23}$ commutes with the adjoints of $\mathfrak{k}_{\Theta}$.  Therefore the $(K_\Theta)_0$ representations on  $V_1$ and $V_3$ and the representations on $V_2$ and $V_4$, are equivalent. One can see that these two different representations are not equivalent.

Assume $J:\mn_\Theta^+\lra \mn_\Theta^+$ is a $K_\Theta$-invariant almost complex structure. The $M$-invariance implies that $V_{[\alpha]}$ is $M$-invariant. For instance, $V_{[\l_1-\l_3]}=\linebreak\lela\{ X_{13},Y_{13},X_{24},Y_{24}\}\rira$ is invariant under $J$. Because of the $\mk_\Theta$ representations described above, we have that $JV_1=V_1$ or $JV_1=V_3$. In the first case, we may have $X_{13}-X_{24}$ as an eigenvalue of $J$, which is not possible, so we obtain $JV_1=V_3$ and $J (X_{13}-X_{24})=a_1(Y_{13}-Y_{24})$ for some $c_1\neq 0$. By analogous reasoning we obtain that $J$ is as follows:
\begin{equation*}
\begin{array}{rcl}
 JY_{12} & = & aY_{12}+cY_{34}, \\ 
 JY_{34} & = & (1+a^2)Y_{12}/c-aY_{34}, \\
 J(X_{13}-X_{24}) &  = & c_1(Y_{13}-Y_{24}),\\ 
J(X_{14}+X_{23}) & = &c_2(Y_{14}+Y_{23}), \\ 
 J(X_{13}+X_{24})  & = & c_3(Y_{13}+Y_{24}),\\ 
 J(X_{14}-X_{23}) & = &  c_4(Y_{14}-Y_{23}),  
\end{array} 
\end{equation*}
where $ c_i, c\neq 0$. But $J\ad_X=\ad_XJ$ for $X\in \mk_\Theta$ implies $c_1=c_4$ and $c_2=c_3$. Direct computations show that this is $M$-invariant and $J\ad_X=\ad_X J$ for all $X\in \mk_\Theta$, therefore, a $K_\Theta$-invariant almost complex structure.

Regarding integrability, it suffices to remark that, for instance, $N_J(Y_{12},X_{13}-X_{24})$ is never zero.
\end{proof}
\smallskip

\begin{proposition} The real flag manifold $\F_\Theta$ of type $D_4$ with $\Theta=\linebreak\{\lambda_1-\lambda_2,\lambda_3+\lambda_4\}$ admits $K_\Theta$ invariant almost complex structures. These structures are not integrable.
  \end{proposition}

\begin{proof} We proceed as in the previous proof. The following is a decomposition on $K_\Theta$ invariant and irreducible subspaces 
$$\mn_\Theta^+=\mgg_{\l_1+\l_2}\oplus \mgg_{\l_3-\l_4}\oplus \sum_{i=1}^4 V_i,$$\vspace{-1cm}
where \begin{multicols}{2}
\begin{eqnarray*}
V_1&=&\langle X_{13}-Y_{24},Y_{14}+X_{23} \rangle \\
V_2&=& \langle X_{13}+Y_{24},Y_{14}-X_{23}\rangle
\end{eqnarray*}

\begin{eqnarray*}
V_3&=&\langle Y_{13}-X_{24},X_{14}+Y_{23} \rangle\\
V_4&=&\langle Y_{13}+X_{24},X_{14}-Y_{23}\rangle
\end{eqnarray*}\end{multicols}

The subspace $V_1$ is $\mathfrak{k}_{\Theta}$-equivalent to subspace $V_3$
and the subspace $V_2$ is $\mathfrak{k}_{\Theta}$-equivalent to subspace $%
V_4 $ through the following linear transformations \linebreak$T_{13}:V_1\longrightarrow V_3$ and $%
T_{24}:V_2\longrightarrow V_4$, given by $T_{13}(X_{13}-Y_{24})=Y_{13}-X_{24}$, $%
T_{13}(Y_{14}+X_{23})=X_{14}+Y_{23}$, $T_{24}(X_{13}+Y_{24})=Y_{13}+X_{24}$ and $%
T_{24}(Y_{14}-X_{23})=X_{14}-Y_{23}$. The other representations are not $\mk_\Theta$ equivalent.

Assume $J$ is a $K_\Theta$-invariant almost complex structure. As before, $JV_1=V_3$ and $JV_2=V_4$ and $J$ verifies 
\begin{equation*}
\begin{array}{rcl}
 JY_{12} & = & aX_{34}+cY_{12},\\ 
  JX_{34} & = & (1+a^2)X_{34}/c-aY_{12}, \\
 J(X_{13}-Y_{24}) &=&c_1(Y_{13}-X_{24}), \\ 
 J(Y_{14}+X_{23}) & =  &  c_2(X_{14}+Y_{23}),\\ 
 J(X_{13}+Y_{24}) & = & c_3(Y_{13}+X_{24}),\\ 
 J(Y_{14}-X_{23}) & =&  c_4(X_{14}-Y_{23}),
\end{array}%
\end{equation*}
$J$ commuting with $\ad_X$, for $X\in \mk_\Theta$ implies $c_1=c_2$ and $c_3=c_4$, and any such $J$ commutes with all $\ad_X\in \mk_\Theta$, so it is $(K_\Theta)_0$ invariant. One can verify that $J$ is also $M$-invariant.

Again, it is possible to see that $N_J(Y_{12},X_{13}-Y_{24})$ never vanishes.
\end{proof}
\smallskip

\begin{proposition}
The real flag manifold $\F_\Theta$ of type $D_4$ with $\Theta=\linebreak\{\lambda_3-\lambda_4,\lambda_3+\lambda_4\}$ admits $K_\Theta$ invariant almost complex structures. These structures are not integrable.
\end{proposition}
\begin{proof}
The following is a decomposition on $(K_\Theta)_0$ invariant and irreducible subspaces 
$$\mn_\Theta^+=\mgg_{\l_1-\l_2}\oplus \mgg_{\l_1+\l_2}\oplus \sum_{i=1}^4 V_i,$$\vspace{-1cm}
where \begin{multicols}{2}
\begin{eqnarray*}
V_1&=&\langle X_{13}+Y_{13},X_{14}+Y_{14} \rangle \\
V_2&=& \langle X_{13}-Y_{13},X_{14}-Y_{14}\rangle
\end{eqnarray*}

\begin{eqnarray*}
V_3&=&\langle X_{23}+Y_{23},X_{24}+Y_{24} \rangle\\
V_4&=&\langle X_{23}-Y_{23},X_{24}-Y_{24}\rangle
\end{eqnarray*}\end{multicols}

The subspace $V_1$ is $\mathfrak{k}_{\Theta}$-equivalent to subspace $V_3$
and the subspace $V_2$ is $\mathfrak{k}_{\Theta}$-equivalent to subspace $%
V_4 $. Indeed, we consider the linear transformations $T_{13}:V_1\longrightarrow
V_3$ given by $T_{13}(X_{13}+Y_{13})=X_{24}+Y_{24}$ and $%
T_{13}(X_{14}+Y_{14})=-(X_{23}+Y_{23})$ and 
 $T_{24}:V_2\longrightarrow V_4$ given by $T_{24}(X_{13}-Y_{13})=X_{24}-Y_{24}$ and $%
T_{24}(X_{14}-Y_{14})=-(X_{23}-Y_{23})$.  

Any $(K_{\Theta})_0$-invariant complex structure $%
J$ is of form 
\begin{equation*}
\begin{array}{rcl}
JX_{12} & = & aX_{12}+cY_{12},\\ 
  JY_{12} & = & (1+a^2)X_{12}/c-aY_{12}, \\
 J(X_{13}+Y_{13}) & = & c_1(X_{24}+Y_{24}), \\ 
 J(X_{14}+Y_{14}) & = & -c_1(X_{23}+Y_{23}),\\ 
 J(X_{13}-Y_{13}) & = & c_2(X_{24}-Y_{24}),\\ 
 J(X_{14}-Y_{14}) & = & -c_2(X_{23}-Y_{23}),
\end{array}%
\end{equation*} Direct computations show that this is also $M$-invariant and therefore $K_\Theta$-invariant. For such structure, $N_J(X_{12},X_{13}+Y_{13})$ never vanishes.
\end{proof}
\medskip

\begin{proposition} The real flag manifolds $\F_\Theta$ of type $D_4$ where $\Theta$ is one of the following sets: \begin{itemize}
\item $\Theta_1=\{\lambda_1-\lambda_2,\lambda_2-\lambda_3,\lambda_3-\lambda_4\}$,
\item $\Theta_2=\{\lambda_1-\lambda_2,\lambda_2-\lambda_3,\lambda_3+\lambda_4\}$, 
\item $\Theta_3=\{\lambda_2-\lambda_3,\lambda_3-\lambda_4,\lambda_3+\lambda_4\}$, 
\end{itemize}
do not admit $K_\Theta$-invariant almost complex structures. 
  \end{proposition}
\begin{proof} Below we give the respective decompositions of $\mn_{\Theta_i}^+$ in $K_\Theta$ invariant and irreducible subspaces. 
\begin{eqnarray*}
\mathfrak{n}_{\Theta_1}^{-}&=&\langle Y_{12}+Y_{34}, Y_{13}-Y_{24},
Y_{14}+Y_{23}\rangle\oplus \langle Y_{12}-Y_{34}, Y_{13}+Y_{24},
Y_{14}-Y_{23} \rangle.\\
\mathfrak{n}_{\Theta_2}^{-}&=&\langle Y_{12}+X_{34}, Y_{13}-X_{24},
X_{14}+Y_{23}\rangle\oplus \langle Y_{12}-X_{34}, Y_{13}+X_{24},
X_{14}-Y_{23} \rangle.\\
\mn_{\Theta_3}^-&=&\langle X_{12}+Y_{12},
X_{13}+Y_{13},X_{14}+Y_{14}\rangle\oplus \langle X_{12}-Y_{12},
X_{13}-Y_{13},X_{14}-Y_{14}\rangle.
\end{eqnarray*}

We see that each of them decomposes as a sum of two irreducible subspaces $V_1$ and $V_2$ which induce non-equivalent representations and such that $\dim V_1=\dim V_2=3$. Lemma \ref{lem2} implies that any $K_\Theta$-invariant complex structure preserves each of these irreducible components, which is not possible since these are odd dimensional. Therefore $\F_{\Theta_i}$ does not admit $K$-invariant almost complex structures for $i=1,2,3$.
\end{proof}


\begin{thebibliography}{99}
\bibitem{grego} A. Arvanitoyeorgos, Geometry of flag manifolds, \textit{Int. J. Geom. Mod. Phys.,}  \textbf{3} (5-6) (2006), 957--974.

\bibitem{arvanitoyeorgos1993new} A. Arvanitoyeorgos, New invariant Einstein metrics on generalized flag manifolds, \textit{Trans. Amer. Math. Soc.}, {\bf 337} (2) (1993), 981--995.

\bibitem{bor} A. Borel, K\"{a}hlerian coset spaces of semi-simple Lie
groups. \textit{Proc. Nat. Acad. of Sci. U.S.A.}. \textbf{40} (1954), 1147--1151.

\bibitem{burstall1990twistor} F. Burstall and J. Rawnsley, \textit{Twistor Theory for Riemannian Symmetric Spaces with Applications to Harmonic Maps of
Riemann Surfaces.} Lecture Notes in Mathematics, 1424. Springer-Verlag, Berlin, 1990.


\bibitem{burstall1987tournaments} F. Burstall and S. Salamon, Tournaments, flags, and harmonic maps, \textit{Math. Ann.} {\bf 277} (2) (1987), 249--265.


\bibitem{cohen20021} N. Cohen, C. Negreiros and L. A. B. San Martin, (1, 2)-symplectic metrics, flag manifolds and tournaments, \textit{Bull. London Math. Soc.}, {\bf 34} (6) (2002), 641--649.


\bibitem{fulton1991representation} W. Fulton and J. Harris,  \textit{Representation Theory: A First Course.} Graduate Texts in Mathematics,129. Readings in Mathematics. Springer-Verlag, 1991.

\bibitem{grama2015invariant} L. Grama, C. Negreiros and
A. Oliveira, Invariant almost complex geometry on flag manifolds: geometric formality and Chern numbers, \textit{ Ann. Mat. Pura Appl. (4)} {\bf 196} (1) (2017), 165--200.


\bibitem{gray1980sixteen} A. Gray and L. Hervella, The sixteen classes of almost Hermitian manifolds and their linear invariants, \textit{Ann. Mat. Pura Appl.}, {\bf 123} (1) (1980), 35--58.

\bibitem{parede} M. Gutierrez, \textit{Aspectos da geometria complexa das variedades flag}, PhD. Thesis, 2000.


\bibitem{helgason1978differential} S. Helgason, \textit{Differential geometry, Lie
groups, and symmetric spaces.} Corrected reprint of the 1978 original. Graduate Studies in Mathematics, 34. American Mathematical Society, Providence, RI, 2001.


\bibitem{knapp2013lie} A. Knapp, \textit{Lie groups beyond an introduction} Progress in Mathematics, 140. Birkhäuser Boston, Inc., Boston, MA, 1996.



\bibitem{upper} S. Kobayashi and K. Nomizu, \textit{ Foundations of Differential Geometry.}, Vol. II, Wiley Interscience, New York, 1969.


\bibitem{Monk} D. Monk, The geometry of flag manifolds, \textit{Proc. London Math. Soc. (3)} \textbf{9} (1959), 253--286.

\bibitem{negreiros1988some} C. Negreiros, Some remarks about harmonic maps into flag manifolds, \textit{Indiana Univ. Math. J.}, {\bf 37} (3)(1988), 617--636.



\bibitem{patrao2015isotropy} M. Patr\~ao and L. A. B. San Martin, The
isotropy representation of a real flag manifold: Split real forms, \textit{Indag. Math. (N.S.)}, {\bf 26} (3) (2015), 547--579.

\bibitem{patrao2012orientability}  M. Patr\~ao, L. A. B. San Martin, L. dos Santos and L. Seco, Orientability of vector bundles over real flag manifolds, \textit{Topology App.}{\bf  159} (10) (2012), 2774--2786.

\bibitem{sangrupos} L. A. B. San Martin, \textit{Grupos de Lie.} Editora UNICAMP, 2016.

\bibitem{san1999algebras} L. A. B. San Martin, \textit{\'Algebras de Lie.} Editora UNICAMP, 1999.

\bibitem{san2003invariant} L. A. B. San Martin and C. Negreiros, Invariant almost Hermitian structures on flag manifolds, \textit{Adv. Math.} \textbf{178} (2) (2003) 277--310.


\bibitem{san2006invariant} L. A. B. San Martin and R. de J. Silva, Invariant nearly-K{\"{a}}hler structures, \textit{ Geom. Dedicata}, {\bf 121} (1) (2006), 143--154.

\bibitem{gw1} J. A. Wolf and A. Gray, Homogeneous spaces defined by Lie group automorphisms I. \textit{J. Diff. Geom.} \textbf{2} (1968), 77--114.

\bibitem{gw} J. A. Wolf and A. Gray, Homogeneous spaces defined by Lie group
automorphisms II. \textit{J. Diff. Geom.} \textbf{2} (1968), 115--159.
\end{thebibliography}
\end{document}